\documentclass[11pt]{amsart}
\usepackage{
    amsmath,
    amsthm,
    amssymb,
    bbm,
    mathrsfs,
    multicol,
    tikz,
    subcaption,
    float,
    color,
    titlesec
    }

\usepackage[margin=1.2in]{geometry}

\usepackage{hyperref}

\usetikzlibrary{calc,cd,knots,graphs,shapes,positioning}

\titleformat{\section}{\large\bfseries}{\thesection}{1em}{}
\titleformat{\subsection}{\bfseries}{\thesubsection}{1em}{}

\newtheorem{theorem}{Theorem}[section]
\newtheorem{proposition}[theorem]{Proposition}
\newtheorem{lemma}[theorem]{Lemma}
\newtheorem{corollary}[theorem]{Corollary}
\theoremstyle{definition}
\newtheorem{example}[theorem]{Example}

\newtheorem{remark}[theorem]{Remark}
\newtheorem{definition}[theorem]{Definition}

\newtheorem{question}[theorem]{Question}
\newtheorem*{acknowledgement}{Acknowledgement}
\numberwithin{equation}{section}

\newcommand{\cat}[1]{\mathcal{#1}}

\newcommand{\RT}{\mathrm{RT}}
\newcommand{\Q}{\mathbb{Q}}
\newcommand{\Gal}{\operatorname{Gal}}
\newcommand{\SL}{\operatorname{SL}}
\newcommand{\id}{\operatorname{id}}
\newcommand{\End}{\operatorname{End}}

\renewcommand{\AA}{\mathcal{A}}
\newcommand{\BB}{\mathcal{B}}
\newcommand{\CC}{\mathcal{C}}
\newcommand{\DD}{\mathcal{D}}
\newcommand{\ZZ}{\mathcal{Z}}
\newcommand{\OO}{\mathcal{O}}

\newcommand{\WW}{\mathcal{W}}
\newcommand{\WQ}{\mathcal{WQ}}
\newcommand{\TY}{\mathcal{TY}}
\newcommand{\HH}{\mathcal{H}}

\newcommand{\BC}{\mathbb{C}}
\newcommand{\BZ}{\mathbb{Z}}
\newcommand{\BR}{\mathbb{R}}
\newcommand{\BQ}{\mathbb{Q}}
\newcommand{\BN}{\mathbb{N}}
\newcommand{\BF}{\mathbb{F}}
\newcommand{\w}{\omega}
\newcommand{\e}{\epsilon}
\newcommand{\ol}{\overline}
\renewcommand{\o}{\otimes}
\renewcommand{\t}{\tau}
\renewcommand{\a}{\alpha}
\newcommand{\s}{\sigma}
\newcommand{\hs}{{\hat{\s}}}
\newcommand{\1}{\mathbbm{1}}
\newcommand{\inv}{^{-1}}
\newcommand{\sVec}{\operatorname{sVec}}
\renewcommand{\Vec}{\operatorname{Vec}}
\newcommand{\Rep}{\operatorname{Rep}}
\newcommand{\ord}{\operatorname{ord}}
\newcommand{\Tr}{\operatorname{Tr}}
\newcommand{\rev}{^{\operatorname{rev}}}
\newcommand{\chialt}{\chi_{\operatorname{alt}}}
\newcommand{\chisym}{\chi_{\operatorname{sym}}}
\newcommand{\jac}[2]{\left(\frac{#1}{#2}\right) }
\renewcommand{\th}[1]{#1^{\operatorname{th}}}

\newcommand{\Qb}{\ol{\BQ}}
\newcommand{\WWp}{\WW_{\operatorname{pt}}}
\newcommand{\lcm}{\operatorname{lcm}}
\newcommand{\FPdim}{\operatorname{FPdim}}

\newcommand{\OnePlusTwoMinus}
{
    \begin{tikzpicture}[baseline=0]
    \node[right = 0.84cm of a] (b) {};
    \draw (b) circle (0.75);
    \node[fill=white,above = 0.5cm of b] {};
    \node[right = 1cm of b,circle,draw] (c) {$\theta^{-1}$};
    \node[left = 0.25cm of b,circle,draw,fill=white] {$\theta$};
    \node[above = 0.5cm of b] (b2) {};
    \node[above = 0.75cm of c] (c2) {};
    \node[below = 1.1cm of b] (b3) {};
    \node[below = 0.75cm of c] (c3) {};
    \node[above = 1.1cm of b] (b4) {};
    \node[below = 0.5cm of b] (b5) {};
    \draw (b4.center) edge (b.center);
    \draw (b5) edge (b.center);
    \draw (b5) edge (b3.center);
    \draw (c2.center) edge (c);
    \draw (c3.center) edge (c);
    \draw[in=90,out=90] (b4.center) edge (c2.center);
    \draw[in=-90,out=-90] (b3.center) edge (c3.center);
    \end{tikzpicture}
}

\newcommand{\StretchCrossing}
{
	\begin{tikzpicture}[baseline=25, x=0.8cm, y=0.8cm,scale=0.8]
	\begin{knot}[flip crossing=2, clip width=4]
	\draw (-1,1.6) node[draw,circle] (A) {$\theta^{-1}$};
	\draw (1, 1.6) node[draw,circle] (B) {$\theta$};
	\strand (-1,0.5)
	to [out=up, in=down] (A.south);
	\strand (-1, 2.5)
	to [out=down, in=up] (A.north);
	\strand (-1, 0.5)
	to [out=down, in=left] (0.5,-0.5)
	to [out=right, in=down] (2,0.5)
	to [out=up, in=down] (2,2.5)
	to [out=up, in=right] (0.5,3.5)
	to [out=left, in=up] (-1,2.5);
	\strand (2.5, 3.5)
	to [out=down, in=up] (B.north);
	\strand (2.5, -0.5)
	to [out=up, in=down] (B.south);
	\strand (-2, 0)
	to [out=down, in=left] (-1, -1)
	to [out=right, in=left] (1.5, -1)
	to [out=right, in=down] (2.5, -0.5);
	\strand (-2, 0)
	to [out=up, in=down] (-2, 3.5)
	to [out=up, in=left] (-1, 4)
	to [out=right, in=left] (1.5, 4)
	to [out=right, in=up] (2.5, 3.5);
	\end{knot}
	\end{tikzpicture}
}

\newcommand{\OneMinusTwoPlus}
{
	\begin{tikzpicture}[baseline=25, x=0.8cm, y=0.8cm,scale=0.8]
	\begin{knot}[flip crossing=2, clip width=4]
	\draw (-1,1.6) node[draw,circle] (A) {$\theta^{-1}$};
	\draw (4, 1.6) node[draw,circle] (B) {$\theta$};
	\strand (A.south)
	to [out=down, in=left] (0.5,-0.5)
	to [out=right, in=down] (2,1.5)
	to [out=up, in=right] (0.3,3.5)
	to [out=left, in=up] (A.north);
	\strand (2.5, 3.5)
	to [out=left, in=up] (1,1.5);
	\strand (2.5, 3.5)
	to [out=right, in=up] (B.north);
	\strand (2.5, -0.5)
	to [out=left, in=down] (1,1.5);
	\strand (2.5, -0.5)
	to [out=right, in=down] (B.south);
	\end{knot}
	\end{tikzpicture}
}

\newcommand{\InnerAction}
{
    \begin{tikzpicture}[baseline=25, x=0.8cm, y=0.8cm, scale=0.8]
    \begin{knot}
    [flip crossing=1, flip crossing=2]
    \strand[only when rendering/.style={dashed}] (-1.3,0.5)
        to [out=left, in=down] (-1.8,1.2)
        to [out=up, in=down] (-1, 2.9);
    \strand[only when rendering/.style={dashed}] (-1.3,0.5)
		to [out=right, in=down] (-1,0.8);
	\draw (-1,1.6) node[draw] (A) {$f$};
	\strand (-1,0.3)
		to [out=up, in=down] (A.south);
	\strand (-1, 3)
		to [out=down, in=up] (A.north);
	\strand (-1, 0.3)
		to [out=down, in=left] (0,-0.7)
		to [out=right, in=down] (1,0.3)
		to [out=up, in=down] (1,3)
		to [out=up, in=right] (0,4)
		to [out=left, in=up] (-1,3);
   	\node at (-0.6,-0.5) [label=left:\footnotesize{$M$}] {};
   	\end{knot}
    \end{tikzpicture}
}

\newcommand{\PassOver}
{
    \begin{tikzpicture}[baseline=25, x=0.8cm, y=0.8cm, scale=0.8]
	\begin{knot}
	[flip crossing=1, flip crossing=2]
	\strand[only when rendering/.style={dashed}] (-1.3,0.5)
	to [out=left, in=down] (-1.8,1.2)
	to [out=up, in=left] (1, 2.9);
	\strand[only when rendering/.style={dashed}] (-1.3,0.5)
	to [out=right, in=left] (1, 2);
	\draw (-1,1.6) node[draw] (A) {$f$};
	\strand (-1,0.3)
	to [out=up, in=down] (-1, 0.4);
	\strand (-1,0.7)
	to [out=up, in=down] (A.south);
	\strand (-1, 2.6)
	to [out=up, in=down] (-1, 3);
	\strand (-1, 2.3)
	to [out=down, in=up] (A.north);
	\strand (-1, 0.3)
	to [out=down, in=left] (0,-0.7)
	to [out=right, in=down] (1,0.3)
	to [out=up, in=down] (1,3)
	to [out=up, in=right] (0,4)
	to [out=left, in=up] (-1,3);
	\node at (-0.6,-0.5) [label=left:\footnotesize{$M$}] {};
	\end{knot}
	\end{tikzpicture}
}

\newcommand{\OuterAction}
{
    \begin{tikzpicture}[baseline=25, x=0.8cm, y=0.8cm, scale=0.8]
	\begin{knot}
	[flip crossing=1, flip crossing=2]
	\strand[only when rendering/.style={dashed}] (0.7,0.5)
	to [out=left, in=down] (0.2,1.2)
	to [out=up, in=down] (1, 2.9);
	\strand[only when rendering/.style={dashed}] (0.7,0.5)
	to [out=right, in=down] (1, 1);
	\draw (-1,1.6) node[draw] (A) {$f$};
	\strand (-1,0.3)
	to [out=up, in=down] (A.south);
	\strand (-1, 3)
	to [out=down, in=up] (A.north);
	\strand (-1, 0.3)
	to [out=down, in=left] (0,-0.7)
	to [out=right, in=down] (1,0.3)
	to [out=up, in=down] (1,3)
	to [out=up, in=right] (0,4)
	to [out=left, in=up] (-1,3);
	\node at (-0.6,-0.5) [label=left:\footnotesize{$M$}] {};
	\end{knot}
	\end{tikzpicture}
}

\newcommand{\SmallLoop}
{
    \begin{tikzpicture}[baseline=25, x=0.8cm, y=0.8cm, scale=0.8]
	\begin{knot}
	[flip crossing=1, flip crossing=2]
	\strand[only when rendering/.style={dashed}] (0.5,0.8)
	to [out=left, in=down] (0.2,1.2)
	to [out=up, in=down] (1, 2.9);
	\strand[only when rendering/.style={dashed}] (0.5,0.8)
	to [out=right, in=down] (0.8, 2.3);
	\draw (-1,1.6) node[draw] (A) {$f$};
	\strand (-1,0.3)
	to [out=up, in=down] (A.south);
	\strand (-1, 3)
	to [out=down, in=up] (A.north);
	\strand (-1, 0.3)
	to [out=down, in=left] (0,-0.7)
	to [out=right, in=down] (1,0.3)
	to [out=up, in=down] (1,3)
	to [out=up, in=right] (0,4)
	to [out=left, in=up] (-1,3);
	\node at (-0.6,-0.5) [label=left:\footnotesize{$M$}] {};
	\end{knot}
	\end{tikzpicture}
}

\newcommand{\SmallDot}
{
	\begin{tikzpicture}[baseline=20, x=0.5cm, y=0.5cm]
	\begin{knot}
	[flip crossing=1, flip crossing=2]
	\strand[only when rendering/.style={dashed}] (0.2,1.2)
	to [out=up, in=down] (1, 2.8);
	\draw (-1,1.6) node[draw] (A) {$f$};
	\draw [fill] (0.2,1.2) circle [radius=0.1];
	\strand (-1,0)
	to [out=up, in=down] (A.south);
	\strand (-1, 3)
	to [out=down, in=up] (A.north);
	\strand (-1, 0)
	to [out=down, in=left] (0,-1)
	to [out=right, in=down] (1,0)
	to [out=up, in=down] (1,3)
	to [out=up, in=right] (0,4)
	to [out=left, in=up] (-1,3);
	\node at (-0.6,-0.5) [label=left:\footnotesize{$M$}] {};
	\end{knot}
	\end{tikzpicture}
}

\newcommand{\TraceOfF}
{
	\begin{tikzpicture}[baseline=20, x=0.5cm, y=0.5cm]
	\begin{knot}
	\draw (-1,1.6) node[draw] (A) {$f$};
	\strand (-1,0)
	to [out=up, in=down] (A.south);
	\strand (-1, 3)
	to [out=down, in=up] (A.north);
	\strand (-1, 0)
	to [out=down, in=left] (0,-1)
	to [out=right, in=down] (1,0)
	to [out=up, in=down] (1,3)
	to [out=up, in=right] (0,4)
	to [out=left, in=up] (-1,3);
	\node at (-0.6,-0.5) [label=left:\footnotesize{$M$}] {};
	\end{knot}
	\end{tikzpicture}
}

\newcommand{\InnerActionTheta}
{
    \begin{tikzpicture}[baseline=25, x=0.8cm, y=0.8cm, scale=0.8]
    \begin{knot}
    [flip crossing=1, flip crossing=2]
    \strand[only when rendering/.style={dashed}] (-1.3,0.5)
        to [out=left, in=down] (-2.1,1.2)
        to [out=up, in=down] (-1, 3.1);
    \strand[only when rendering/.style={dashed}] (-1.3,0.5)
		to [out=right, in=down] (-1,0.8);
	\draw (-1,1.6) node[draw] (A) {$\theta^{-1}$};
	\strand (-1,0.3)
		to [out=up, in=down] (A.south);
	\strand (-1, 3)
		to [out=down, in=up] (A.north);
	\strand (-1, 0.3)
		to [out=down, in=left] (0,-0.7)
		to [out=right, in=down] (1,0.3)
		to [out=up, in=down] (1,3)
		to [out=up, in=right] (0,4)
		to [out=left, in=up] (-1,3);
   	\node at (-0.6,-0.5) [label=left:\footnotesize{$M$}] {};
   	\end{knot}
    \end{tikzpicture}
}

\title{Higher Gauss sums of modular categories}
\author{Siu-Hung Ng}
\address{Department of Mathematics,
Louisiana State University,
Baton Rouge, LA 70803, USA}
\email{rng@math.lsu.edu}
\thanks{The first author was partially supported by NSF DMS1664418}

\author{Andrew Schopieray}
\address{School of Mathematics and Statistics,
UNSW Sydney,
NSW 2052,
Australia}
\email{a.schopieray@unsw.edu.au}

\author{Yilong Wang}
\address{Department of Mathematics,
Louisiana State University,
Baton Rouge, LA 70803, USA}
\email{yilongwang@lsu.edu}
\begin{document}
\begin{abstract}
The definitions of the $\th{n}$ \emph{Gauss sum} and the associated $\th{n}$ \emph{central charge} are introduced for premodular categories $\mathcal{C}$ and $n\in\mathbb{Z}$. We first derive an expression of the $\th{n}$ Gauss sum of a modular category $\CC$, for any integer $n$ coprime to the order of the T-matrix of $\CC$, in terms of the first Gauss sum, the global dimension, the twist and their Galois conjugates. As a consequence, we show for these $n$, the higher Gauss sums are $d$-numbers and the associated central charges are roots of unity.  In particular, if $\mathcal{C}$ is the Drinfeld center of a spherical fusion category, then these higher central charges are 1.  We obtain another expression of higher Gauss sums for de-equivariantization and local module constructions of appropriate premodular and modular categories.  These expressions are then applied to prove the Witt invariance of higher central charges for pseudounitary modular categories.
\end{abstract}

\maketitle

\begin{section}{Introduction}
In 1801 \cite{GaussDA}, C.~F.~Gauss introduced the sum $\t_n(k)=\sum_{j=0}^{k-1}e^{2\pi i n j^2/k}$, which is now called a (\emph{quadratic}) \emph{Gauss sum}. The value of this sum was computed by Gauss up to a sign in 1805 for $\t_1(k)$ when $k$ is odd; it is equal to $\pm \sqrt{k}$ or $\pm i\sqrt{k}$ depending on whether $k \equiv 1$ or $3 \pmod{4}$, respectively. The sign of this quadratic Gauss sum was finally proved to be positive in 1811 \cite{sumatio}.

Another version of Gauss sum was introduced in 1840 by L.~G.~Dirichlet \cite{Dir}: for each multiplicative character $\chi$ of a field $\BF_p$ of odd prime order $p$, its Gauss sum is $g_n(\chi)=\sum_{j=0}^{p-1}\chi(j)e^{2 \pi i nj/p}$. The quadratic Gauss sum $\t_n(p)$ can be recovered from $g_n(\chi)$  for any integer $n$ coprime to $p$ if $\chi$ is the Legendre symbol of $\BF_p$. Moreover, the Gauss sums $g_n(\chi)$ also reveal a relation with Frobenius-Schur indicators via Jacobi sums when counting the number of solutions of the equation $x^m+y^m=1$ in $\BF_p$ (cf. \cite{Ireland}). The narrative of higher Gauss sums of premodular categories in this paper is a generalization of quadratic Gauss sums.

Quadratic Gauss sums can be understood in terms of the sum of the values of quadratic forms of finite abelian groups. In particular, if one considers the quadratic form $q: \BZ_k \to \BC^\times$ with $q(j)=e^{2 \pi i n j^2/k}$, then the sum of the values of $q$ is the Gauss sum $\t_n(k)$.

Premodular or ribbon categories are categorical generalizations of quadratic forms of finite abelian groups. Abelian groups equipped with \emph{nondegenerate} quadratic forms are the classical counterparts of modular categories, which arise naturally in low-dimensional topology and rational conformal field theory. Moreover, modular categories constitute the mathematical foundation of topological quantum computations \cite{WZ, RW}. One can assign a real number called the \emph{quantum dimension} to each object of a ribbon category $\CC$. There is a natural isomorphism $\theta$ of the identity functor of $\CC$, called the \emph{ribbon isomorphism}, whose values on the simple objects of $\CC$ play the role of the values of a quadratic form of a finite abelian group.

In the  late $\th{20}$ century,   a notion of \emph{Gauss sum}  incidentally emerged in quantum invariants of $3$-manifolds derived from modular categories introduced by Reshetikhin and Turaev (cf. \cite{TuraevBook, witten, rt, tur, lick, kirby, murakami}).   The modulus of a Gauss sum $\t$ and the finiteness of the order of $\t/\ol \t$ reemerge in the contexts of invariants of $3$-manifolds and rational conformal field theory (RCFT) (cf. \cite{TuraevBook} and \cite{vafa}). Similar to its classical counterpart, the Gauss sums characterize modular categories with T-matrices of order 2, up to equivalence \cite{WW}. Moreover, the classification of relations among the Witt classes $[\mathcal{C}(\mathfrak{sl}_2,k)]$ \cite{DNO} and $[\mathcal{C}(\mathfrak{sl}_3,k)]$ \cite{schopieray2017} were proven in part using first central charge arguments.

In this paper, we define \emph{higher Gauss sums} $\tau_n$ and \emph{anomalies} $\a_n$ of a premodular category as generalizations of the Gauss sum of a quadratic form of a finite abelian group $G$ and the Jacobi symbol $\jac{-1}{|G|}$. A natural choice of the square root of the anomaly $\a_n$ is defined as the \emph{higher} (\emph{multiplicative}) \emph{central charge} $\xi_n$ for $n\in\mathbb{Z}$ which is motivated by rational conformal field theory.

Frobenius-Schur indicators are arithmetic invariants of premodular categories. They were first introduced for the representations of finite groups a century ago. Their recent generalizations to Hopf algebras \cite{LM00} and rational conformal field theory \cite{Bantay} inspired the development of Frobenius-Schur indicators of pivotal categories \cite{NSind}.  Similar to their classical counterparts, the higher Gauss sums of modular categories are closely related to the Frobenius-Schur indicators.  In particular, the modulus of the higher Gauss sums of a modular category are completely determined by the Frobenius-Schur indicators. More relations between these arithmetic invariants and examples are demonstrated in Section \ref{sec:def}.

One subtlety to the definitions of higher anomaly or central charges for a premodular category is that they are only well-defined when $\tau_n(\mathcal{C})\neq0$.  For a modular category $\CC$, $\tau_1(\mathcal{C})\tau_{-1}(\mathcal{C})=\dim(\mathcal{C})\neq0$ (cf. \cite{TuraevBook, BakalovKirillov, MugerSubfactor2, ENO}).  Our Theorem \ref{thm:Galois} shows that for any modular category $\mathcal{C}$, $\tau_n(\mathcal{C})\neq0$ when $n$ is coprime to the order of the T-matrix of $\CC$.

A $d$-\emph{number} is an algebraic integer which generates in the ring of algebraic integers an ideal fixed by the absolute Galois group \cite{codegrees}. The formal codegrees of spherical fusion categories are $d$-numbers and they have been used to prove the non-existence of fusion categories associated with specific fusion graphs/rules (cf. \cite{jonessub, Calegari2011, penneys}). In
Corollary \ref{cor:RootOfUnity}, higher Gauss sums $\t_n(\CC)$ are shown  to be $d$-numbers and the higher central charge $\xi_n(\CC)$ are root of unity under the assumption that $\CC$ is modular and $n$ is coprime to the order of the T-matrix of $\CC$.  In particular, when $\CC$ is the Drinfeld center of a spherical fusion category, the coprime higher central charges are 1 (Theorem \ref{thm:GaussSumOfCenter}), which may not hold for other $n$ by Example \ref{ex:shi}.

\par Higher Gauss sums and central charges behave well under standard constructions such as de-equivariantizations of premodular categories and categories of local modules over certain connected \'etale algebras (also known as simple current extensions or condensations). In particular, Theorem \ref{th:deeq} proves that $|G|\tau_n(\CC^0_G)=\tau_n(\mathcal{C})$ and $\xi_n(\CC^0_G)=\xi_n(\mathcal{C})$ where $\CC^0_G$ is the condensation of the premodular category $\CC$ by a \emph{Tannakian subgroup} $G$ of $\CC$ for any integer $n$ coprime to the order of the T-matrix of $\CC$.   In addition, if $\CC$ is modular (Theorem \ref{th:deek}) the preceding statement holds for the category  $\CC_A^0$ of local modules  over a \emph{ribbon algebra} $A$ of $\CC$ (defined in Definition \ref{def:A-module}).

As an application of Sections \ref{sec:arith} and \ref{sec:deq}, we prove in Theorem \ref{thm:Witt} that two Witt equivalent pseudounitary modular categories must have the same higher central charge $\xi_n$ for any integer $n$ coprime to the orders of their T-matrices.
We use this result to distinguish Witt equivalence classes of pseudounitary modular categories which are indistinguishable using the first  central charge alone. The higher central charges which are well-defined for any two Witt equivalent pointed modular categories are invariants, and we conjecture this statement could be generalized for all Witt equivalence classes.

The organization of this paper is as follows:  Section \ref{prelims} describes the notation, basic definitions, and necessary concepts while introducing fundamental examples. Section \ref{sec:def} motivates our definitions of higher Gauss sums and higher (multiplicative) central charges for premodular categories. The relations between higher Gauss sums and Frobenius-Schur indicators, and a relation between the first and second Gauss sums are shown in this section. A broad range of examples are given to illustrate their properties.  Section \ref{sec:arith} describes the Galois action on the modular data of modular categories which is the key to prove our main result Theorem \ref{thm:Galois}.  Section \ref{sec:deq} consists of a sequence of technical lemmas proven by graphical calculus. These lemmas and  Theorem \ref{thm:Galois} are then applied to prove Theorems \ref{th:deeq} and \ref{th:deek}.  Some examples of computing higher Gauss sums and central charges are illustrated.  Section \ref{sec:witt} proves the Witt invariance of certain higher central charges (Theorem \ref{thm:Witt}). Applications of higher central charges to differentiate Witt equivalence classes which are indistinguishable by the first central charge alone are demonstrated.  The paper ends with some open questions to stimulate conversation and future research.
\end{section}


\begin{section}{Preliminaries}\label{prelims}
\begin{subsection}{Premodular and modular categories}
In this section, we recall some basic definitions and results on fusion and modular categories. The readers are referred to \cite{tcat} and \cite{MugerSubfactor2} for the details.

Throughout this paper, a \emph{fusion category} $\CC$ is a $\mathbb{C}$-linear, abelian, semisimple, rigid monoidal category with finite-dimensional Hom-spaces, finitely-many isomorphism classes of simple objects and simple monoidal unit $\1$. In particular,  we abbreviate the dimension of the Hom-space $ \CC(X, Y) $ over $\BC$ by
$$
[X:Y]_\CC := \dim_\BC(\CC(X, Y))
$$
for any $ X, Y \in \CC $, and  the set of isomorphism classes of simple objects of $\CC$ by $\OO(\CC)$.

The duality of $\CC$ can be extended to a contravariant monoidal equivalence $(-)^*$, and so  $(-)^{**}$ defines a monoidal equivalence on $\CC$. A \emph{pivotal structure} on $\CC$ is a natural isomorphism $j: \id_\CC \to (-)^{**}$ of monoidal functors. For any given pivotal structure $j$ on $\CC$, one can define a (left) \emph{trace} $\Tr_\CC(f) \in \BC$ for any endomorphism $f: V \to V$ in
$\CC$ (see for example \cite{NgSchaunburgSpherical}). In particular, $\Tr_\CC(\id_V)$ is called the (left pivotal) \emph{dimension} of $V$ and denoted by $\dim_\CC(V)$, or $ \dim(V) $ when the category is clear from the context.  The pivotal structure $j$ is said to be \emph{spherical} if $\dim(V)=\dim(V^*)$ for all $V  \in \OO(\CC)$. In this case, $\dim (V)$ is a real cyclotomic integer for $V \in \CC$ (cf. \cite{ENO}). In particular, $\dim(V)$ is totally real. The \emph{global dimension} $\dim (\CC)$ of $\CC$ is given by
$$
\dim (\CC) = \sum_{V \in \OO(\CC)} \dim (V)^2.
$$
Note that $\dim(\CC)$ can be defined without using any pivotal structure of $\CC$ \cite{MugerSubfactor2, ENO}. Recall that an algebraic number $ a $ is called \emph{totally positive} if every Galois conjugate of $ a $ is a positive real number. Since $\dim(V)$ is totally real, $ \dim(\CC) $ is a totally positive cyclotomic integer.

A fusion category $\CC$ is called \emph{pseudounitary} if $\dim(\CC)$ agrees with $\FPdim(\CC)$, the Frobenius-Perron dimension of $\CC$. In this case, $\CC$ admits a unique spherical pivotal structure such that $\dim(V)$ is the Frobenius-Perron dimension of $V$ for all $V \in \CC$ (cf. \cite{ENO}). Throughout this paper, we assume that any pseudounitary fusion category is equipped with such a canonical spherical pivotal structure.

By virtue of \cite[Theorem 2.2]{NSind} we may assume without loss of generality that any pivotal category $\CC$ in this paper is strict. In other words, $\CC$ is a strict monoidal category such that $(-)^*$ is a strict monoidal functor, $(-)^{**}=\id_\CC$, and the pivotal structure $j:\id_\CC \to (-)^{**}$ is the identity. Under this assumption, we can perform computations using graphical calculus with the conventions of \cite{BakalovKirillov, KiO} for any premodular category $ \CC $.

For a  braided fusion category $\CC$ with braiding $c$, the \emph{M\"uger centralizer} or \emph{relative commutant} $\DD'$ of a full fusion subcategory $\DD$ of $\CC$ is the full subcategory generated by
$$\{X \in \OO(\CC)\mid c_{Y,X} \circ c_{X,Y}=\id_{X\o Y} \text{ for all } Y \in \OO(\DD)\} .$$
Note that $\DD'$ is always a fusion subcategory of $\CC$. A braided fusion category is said to be \textit{nondegenerate} if $\OO(\CC') = \{\1\}$. Note that the \emph{Drinfeld center} $\mathcal{Z}(\mathcal{C})$ of a fusion category $\CC$ is a nondegenerate braided fusion category.

A \emph{premodular category} $\CC$ is a braided spherical fusion category. The (unnormalized) S-\emph{matrix} of a premodular category $\CC$ is defined as
$$
S_{X,Y} = \Tr_\CC(c_{Y,X^*}\circ c_{X^*,Y}) \quad\text{for }X, Y \in \OO(\CC)\,.
$$
In particular, $S_{\1,\1} =1$. Nondegeneracy of the braiding of any premodular category $\mathcal{C}$ can be characterized by the invertibility of its S-matrix (cf. \cite{MugerSubfactor2}). In this case, $\CC$ is called a \emph{modular category}.

If  $\DD$ is a spherical fusion category, then its Drinfeld center $\ZZ(\DD)$ inherits the spherical structure of $\DD$, and hence a modular category. Throughout this paper, the Drinfeld center $\ZZ(\DD)$ of a spherical fusion category $\DD$ is always assumed to be spherical with the inherited spherical structure from $\DD$.

In a premodular category $\CC$, the underlying braiding $c$ determines the Drinfeld isomorphism $u:\id_\CC \to (-)^{**}$ of $\BC$-linear functors (cf. \cite[p38]{NgSchaunburgSpherical}). The \emph{twist} or \emph{ribbon} isomorphism is defined as $\theta = u\inv j$ where $j$ is the spherical structure of $\CC$. For each $X\in\mathcal{O}(\mathcal{C})$, $\theta_X=\lambda\cdot\text{id}_X$ for some scalar $\lambda\in\mathbb{C}$. We will use the abuse notation to denote $\lambda$ by $\theta_X$. By  \cite{andem,vafa}, $\theta_X$ is a root of unity and so the T-\emph{matrix} of $\CC$, which is defined as
$$
T_\CC =(\delta_{X,Y} \theta_X)_{X,Y \in \OO(\CC)}\,,
$$
has finite order. The order $N=\ord(T_{\ZZ(\CC)})$ is called the \emph{Frobenius-Schur exponent} of $\CC$ (cf. \cite{NgSchaunburgSpherical}), which is generally greater than $\ord(T_\CC)$. Since the S-matrices of $\CC$ and $\ZZ(\CC)$ are defined over the cyclotomic field $\BQ(e^{2 \pi i/N})$  (see \cite{NS10}), $N$ is also the \emph{conductor} of $\CC$. If $\CC$ is modular, then $\ord(T_\CC) = \ord(T_{\ZZ(\CC)})$. Therefore, we will simply call $\ord(T_\CC)$ the conductor or the FS-exponent of $\CC$ when it is modular.

If $\mathcal{C}$ is a premodular category with braiding $c$, then we define $\mathcal{C}^\text{rev}$ to be identical to $\mathcal{C}$ as spherical fusion categories but with the reversed braiding $\tilde{c}_{X,Y}:=c_{Y,X}^{-1}$ for all $X,Y\in\mathcal{C}^\text{rev}$.  In this case the associated twist $\tilde{\theta}$ satisfies $\tilde{\theta}_X=\theta^{-1}_X$ for all $X\in\mathcal{C}^\text{rev}$.

 For any finite group $G$, the category $\Rep(G)$ of finite-dimensional complex representations of $G$ is a braided fusion category with the braiding inherited from $\Vec$, the category of finite-dimensional $ \BC $-spaces. In particular, $\CC=\Rep(G)$ is \emph{symmetric}, i.e., $\OO(\CC') = \OO(\CC)$. Moreover, $\Rep(G)$ is pseudounitary and the corresponding pivotal dimensions coincide with the dimensions of vector spaces over $\BC$. Therefore, $\Rep(G)$ admits a natural premodular category structure and its ribbon isomorphism $\theta$ is the identity. For the purpose of this paper, a  premodular category $\CC$ is called \emph{Tannakian}  if $\CC$ is equivalent to $\Rep(G)$ for some finite group $G$.   In particular, the ribbon isomorphism of a Tannakian category is the identity natural isomorphism.

\begin{example}\label{ex:group}
Let $A$ be a finite abelian group and $q:A\to\mathbb{C}^\times$ a quadratic form. By the results of \cite{EMa, EMb}, there exists an Eilenberg-MacLane 3-cocycle $ (\w, c) $ of $ A $, where $ \w $ is a 3-cocycle and $ c $ is a 2-cochain of $ A $ such that $c(a,a) =q(a)$ for all $a \in A$. In particular, $c$ defines a braiding on the fusion category $\Vec_A^\w$, the category of finite-dimensional $A$-graded vector spaces over $\BC$ with the associativity isomorphism given by $\w$. The fusion category $\Vec_A^\w$ is pseudounitary with the canonical pivotal dimensions given by the usual dimension of $ \BC $-spaces. Therefore, $\Vec_A^\w$ with the braiding $c$ is a premodular category and we denote it by $\Vec_A^{(\w,c)}$. The quadratic form $q$ completely determines the equivalence class of  $\Vec_A^{(\w,c)}$. We denote by $\mathcal{C}(A,q)$ any of these equivalent premodular categories.  This premodular category $\mathcal{C}(A, q)$ is modular if and only if the corresponding quadratic form $q$ is nondegenerate. The category of super vector spaces $\sVec$ can be defined as the premodular category $\CC(\BZ_2, q)$ with $q(1)=-1$.
\end{example}

\begin{example}[premodular categories from Lie theory]
Another family of premodular categories can be realized by a construction based on the representation theory of the $q$-deformed universal enveloping algebra $\mathcal{U}_q(\mathfrak{g})$ (cf. \cite{lusztig, postriksemisimple}) for a complex finite-dimensional simple Lie algebra $\mathfrak{g}$ and a complex parameter $q$.  These categories have a long history in mathematical physics. The readers are referred to \cite{rowell2006} for a general survey of the subject. The properties of the resulting categories depend heavily on $q$.  For certain roots of unity, modular tensor categories are produced while other choices may result in premodular categories, or ones with infinitely-many isomorphism classes of simple objects.  For the illustrative examples in this paper we will only consider those roots of unity of the form $q=e^{\pi i/(m(k+h^\vee))}$ where  $k\in\BN$ is the \emph{level},  $h^\vee$ is the dual Coxeter number of $\mathfrak{g}$ and $m \in \{1,2,3\}$ is a scaling factor dependent on $\mathfrak{g}$.  The categories $\mathcal{C}(\mathfrak{g},k)$ in this smaller collection are known to be unitary modular tensor categories with modular data accessible via the Kac-Petersen formulas \cite{kac}.
\end{example}

Given two premodular categories $\mathcal{C}$ and $\mathcal{D}$, their \emph{(Deligne) tensor product} $\mathcal{C}\boxtimes\mathcal{D}$ has simple objects $X\boxtimes Y$ for $X\in\mathcal{O}(\mathcal{C})$ and $Y\in\mathcal{O}(\mathcal{D})$, and is again premodular with the ribbon structure given by  $\theta_{X\boxtimes Y}=\theta_X \o \theta_Y$.  For any braided fusion category $\mathcal{C}$,  it follows from a well-known result of M\"uger \cite{MugerSubfactor2} that
\begin{equation}\label{eq:center}
\mathcal{Z}(\mathcal{C})\simeq\mathcal{C}\boxtimes\mathcal{C}^\text{rev}
\end{equation}
is a braided equivalence if and only if $\mathcal{C}$ is nondegenerate.

\end{subsection}

\begin{subsection}{Local modules and the Witt group}\label{sec:local}

There are several constructions of new modular categories from a given one in the context of rational conformal field theory.  One method is to consider a certain subcategory of the tensor category of modules over a commutative algebra object in a modular category. The readers are referred to \cite{KiO} for more details.  When one adds various restrictions to the commutative algebras under consideration, the resulting tensor category is fusion, braided, premodular, and sometimes modular.  Here we review the basics of this theory assuming the reader is familiar with the rudimentary definitions of modules over algebras in fusion categories which can be found in \cite{KiO} and \cite[Chapter 8]{tcat} when needed.

\begin{definition} \label{d:ribbon_alg}
	Let $\mathcal{C}$ be a premodular category.  An algebra $A$ of $\CC$ (or simply $ A \in \CC $) is called \emph{connected \'etale} if the multiplication morphism $m:A\otimes A\to A$ satisfies the following conditions:
    \begin{itemize}
        \item[(a)] $[\1 : A]_\CC=1$ (\emph{connected}),
        \item[(b)] $m\circ c_{A,A}=m$, (\emph{commutative}) and
        \item[(c)] $m$ splits as a morphism of $A$-bimodules (\emph{separable}).
    \end{itemize}
	A connected \'etale algebra $ A $ of $ \cat{C} $ is called \emph{ribbon} if
	\begin{itemize}
        \item[(d)] $\dim A \ne 0$ and $\theta_A=\id_A$.
	\end{itemize}
\end{definition}

The conditions (a), (b), (c) and (d) for a ribbon algebra ensure the category of left $A$-modules $\mathcal{C}_A$ is a spherical fusion category equipped with the spherical structure inherited from $\mathcal{C}$ (cf. \cite{KiO}). However, the braiding of $\mathcal{C}$ can only be passed onto  some fusion subcategories of $\mathcal{C}_A$ in general.  The following definition is due to Pareigis \cite{pareigis}.

\begin{definition}\label{def:A-module}
    Let $\mathcal{C}$ be a premodular category with a connected \'etale algebra $A\in\mathcal{C}$.  Denote by $\mathcal{C}_A^0$ the full subcategory of \emph{local} (or \emph{dyslectic}) $M\in \mathcal{C}_A$, i.e.,
    \begin{equation}
        \rho\circ c_{M,A} \circ c_{A,M}=\rho,
    \end{equation}
     where $\rho:A\otimes M\to M$ is the left $A$-module action of $M$.
\end{definition}

It is shown in \cite{KiO} that if a connected \'etale algebra $A$ is in addition ribbon, then $\mathcal{C}_A^0$ is premodular. Moreover, $ \CC_A^0 $ is modular if $\mathcal{C}$ is \cite[Theorem 4.5]{KiO}. It is worth to note that the same results could be established for the category of right local modules of $A$.

For a ribbon algebra $A\in\mathcal{C}$ we denote the dimension of a left $A$-module $M\in\mathcal{C}_A$ by $\dim_A(M)$.  By \cite[Theorem 1.18]{KiO}, we have the relations
\begin{equation}\label{eq:Dimensions}
    \dim_A(M)=\dfrac{\dim_\CC(M)}{\dim_\CC(A)},\qquad\text{and}\qquad\dim(\mathcal{C}_A^0)
    =\dfrac{\dim(\mathcal{C})}{\dim_\CC(A)^2}.
\end{equation}

Let $\DD$ be a Tannakian subcategory of $\CC$. Suppose $\DD$ is equivalent to $\Rep(H)$ for some finite group $H$ as premodular categories. The dual group algebra $A=\BC[H]^* \in \Rep(H)$, called the \emph{regular algebra} of $\DD$, is a ribbon algebra of $\Rep(H)$, and hence of $\CC$.

\begin{example}\label{ex:reg}
    Let $q$ be a nondegenerate quadratic form on a finite abelian group $G$. Then $\mathcal{C}(G,q)$ is a modular category of rank $|G|$.  Assume in addition that the metric group $(G,q)$ has an isotropic subgroup $H\subset G$ (i.e. $q(x)=1$ for all $x\in H$).  Then $\mathcal{C}(G,q)$ contains a Tannakian subcategory $\DD$ equivalent to $\mathrm{Rep}(H)$ as premodular categories. Let $A$ be the regular algebra of $\DD$. Then  $\dim(\mathcal{C}(G,q)_A^0)=|G|/|H|^2$.
\end{example}

Suppose the premodular category $\CC$ contains a Tannakian subcategory $\DD$ which is equivalent to $\Rep(H)$ for a finite group $H$ as premodular categories. Then $\DD \subset \DD'$ and $ \DD $ is a Tannakian subcategory of $\DD'$. We denote by $(\DD')_H$ the \emph{de-equivariantization} of $\DD'$ by $H$. Let $A$ be the regular algebra of $\DD$. By \cite[Proposition 4.56(i)]{DGNO}, the local $A$-module category $\CC_A^0$ is equivalent to  $(\DD')_H$ as premodular categories.

Another application of local module categories lies in the study of the \textit{Witt equivalence} of nondegenerate braided fusion categories introduced in \cite{DMNO}.

\begin{definition}
Nondegenerate braided fusion categories $\mathcal{C}$ and $\mathcal{D}$ are \emph{Witt equivalent} if there exist fusion categories $\mathcal{A}_1,\mathcal{A}_2$ such that $\mathcal{Z}(\mathcal{A}_1)\boxtimes\mathcal{C}\simeq\mathcal{Z}(\mathcal{A}_2)\boxtimes\mathcal{D}$ is a braided equivalence. The Witt equivalence class of $\CC$ is denoted by $[\CC]$.
\end{definition}

Two nondegenerate braided fusion categories $\CC$ and $\DD$ are Witt equivalent if there exist connected \'etale algebras $A\in\mathcal{C}$ and $B\in\mathcal{D}$ such that $\mathcal{C}_A^0\simeq\mathcal{D}_B^0$ as braided fusion categories (cf. \cite[Proposition 5.15]{DMNO}).

The  Witt equivalence classes of nondegenerate braided fusion categories form an abelian \emph{Witt group} $\mathcal{W}$ under the Deligne tensor product with $[\mathcal{C}]^{-1}=[\mathcal{C}^\mathrm{rev}]$ (see Equation \ref{eq:center}). The Witt group $\WW$  of nondegenerate braided fusion categories is a generalization of the Witt group $\WQ$ of nondegenerate quadratic forms of finite abelian groups. In particular, the subgroup $\WWp(p)$ of $\WW$ generated by the Witt equivalence classes of the pointed modular categories $\CC(G,q)$, where $p$ is a prime and $G$ is a finite abelian $p$-group, is canonically isomorphic to the classical Witt group $\WQ(p)$ of nondegenerate quadratic forms of  finite abelian $p$-groups  (see \cite[Section 5.3]{DMNO}).
\end{subsection}
\end{section}


\begin{section}{Higher Gauss sums and central charges}\label{sec:def}
	
\begin{subsection}{Definition and motivation}
	
Let $\CC$ be a modular category. The Gauss sums $\tau^{\pm}(\CC)$  of $\CC$ are defined as
\begin{equation}\label{class}
\tau^\pm(\mathcal{C})=\sum_{X\in\mathcal{O}(\mathcal{C})}\theta_X^{\pm1}\dim(X)^2.
\end{equation}
Their properties can be found in the literature such as  \cite{BakalovKirillov}, \cite{MugerSubfactor2}, \cite{tcat} and \cite{dong2015}. It is well-known that $\tau^+(\mathcal{C})\tau^-(\mathcal{C})=\dim(\mathcal{C})\ge 1$, and $\frac{\tau^+(\CC)}{\tau^-(\CC)}$ is a root of unity \cite{vafa}. In particular,
$|\tau^+(\mathcal{C})|=\sqrt{\dim (\CC)}$.  From this, the  definition of the (multiplicative) central charge of $\mathcal{C}$ can be defined as $\xi(\mathcal{C}):=\tau^+(\mathcal{C})/\sqrt{\dim(\mathcal{C})}$, which is also a root of unity.

The same definition \eqref{class} of Gauss sums $\tau^{\pm}(\CC)$ can be extended easily to a premodular category $\mathcal{C}$.  However, $\tau^{\pm}(\CC)$  could be zero which can be demonstrated in
the category  $\sVec$ of super vector spaces where  $\tau^+(\mathrm{sVec})=\tau^-(\mathrm{sVec})=1-1=0$ (cf. Example \ref{ex:group}).  Therefore, $\xi(\mathrm{sVec})$ is undefined with this definition.  Nevertheless, the notion of higher of Gauss sums can be generalized and studied for arbitrary premodular categories.

\begin{definition}\label{def:gausssum}
	Let $\mathcal{C}$ be a premodular category and $n\in\mathbb{Z}$.  The \emph{$\th{n}$ Gauss sum} of $\mathcal{C}$ is defined as
	\begin{equation*}
	\tau_{n}(\mathcal{C}):=\sum_{X\in\mathcal{O}(\mathcal{C})}\theta^n_X\dim(X)^2.
	\end{equation*}
\end{definition}

Note that $\overline{\tau_n(\mathcal{C})}=\tau_{-n}(\mathcal{C})$ as $\dim(X)\in\mathbb{R}$ for all $X\in\OO(\CC)$ (cf. \cite{ENO}).
\subsection{Gauss sums and Frobenius-Schur indicators}
For any object $X$ in a spherical fusion category $\CC$ and $n \in \BZ$, the $\th{n}$ \emph{Frobenius-Schur indicator} (FS-indicator) of $X$, denoted by $\nu_n(X)$,  is a scalar introduced in \cite{NSind, NS10}. The notion was also previously defined in the contexts of Hopf algebras and rational conformal field theory (cf. \cite{LM00}, \cite{Bantay}). We will simply define the $\th{n}$ FS-indicator of $\CC$ as
\begin{equation}
\nu_n(\CC):=\sum_{X\in\OO(\CC)}\dim(X)\nu_n(X)\,.
\end{equation}
\begin{remark} \label{r:qH}
	By virtue of \cite[Corollary 5.6]{NS10}, $\nu_{-n}(X)= \ol{\nu_n(X)}$ for $n \in \BZ$ and $X \in \CC$. Therefore, $\nu_{-n}(\CC)=\ol{\nu_n (\CC)}$ for $n \in \BZ$. Moreover, if $\CC = \Rep(H)$ for any semisimple quasi-Hopf algebra $H$ over $\BC$, then  $\nu_n(\CC) = \nu_n(H)$ where $H$ is considered as the regular representation of $H$.
\end{remark}

If $\CC$ is modular,  the modulus of $\tau_n(\CC)$ can be expressed in terms of $\nu_n(\CC)$ as stated in the following proposition, which is essentially proved in \cite[Proposition 5.5]{KMN} with a  different emphasis in the statement.
\begin{proposition}{\rm (\cite[Proposition 5.5]{KMN})} \label{p:gauss_sum}
	Let $\CC$ be a modular category. Then for any integer $n$, we have
	$$
	|\tau_n(\CC)|^2 = \dim(\CC)\nu_n(\CC)\,.
	$$
	In particular, $\dim(\CC)$ divides $|\tau_n(\CC)|^2$ in the ring of algebraic integers, and $\nu_n(\CC)$ is a totally non-negative cyclotomic integer for any $n \in \BZ$.\qed
\end{proposition}

The preceding proposition can be further refined if  $\CC$ is a Drinfeld center. The following refinement is essentially proved in \cite[Lemma 3.8]{ShimizuComputation}.
\begin{proposition}{\rm (\cite[Lemma 3.8]{ShimizuComputation})} \label{p:center_gauss_sum1}
	Let $\DD$ be a spherical fusion category and $\CC =\ZZ(\DD)$. Then, for $n \in \BZ$,
	\begin{equation*}
	\tau_n(\mathcal{\CC}) = \dim (\DD) \nu_n(\DD). \qed
	\end{equation*}
\end{proposition}
Thus, we have the immediate corollary for the case of semisimple quasi-Hopf algebras.
\begin{corollary} \label{c:center_gauss_sum2}
	Let $\CC=\ZZ(\Rep(H))$  for some semisimple quasi-Hopf algebra $H$. Then for any
$n \in \BZ$,
	\begin{equation}\label{eq:prop1eq2}
	\tau_n(\mathcal{C})=\dim(H)\nu_n(H).
	\end{equation}
	If in  addition $H$ is a Hopf algebra, then
	\begin{equation}\label{eq:prop1eq3}
	\tau_2(\mathcal{C})=\dim(H) \Tr(S_H) \in \BZ
	\end{equation}
	where $S_H$ is the antipode of $H$.
\end{corollary}
\begin{proof}
	Equation (\ref{eq:prop1eq2}) follows directly from Remark \ref{r:qH} and Proposition \ref{p:center_gauss_sum1}, and Equation (\ref{eq:prop1eq3}) is an immediate consequence of (\ref{eq:prop1eq2}) and \cite[Theorem 2.7(3)]{LM00}.
\end{proof}

Now, we can apply  Corollary \ref{c:center_gauss_sum2} to compute the higher Gauss sums   of the Drinfeld double any finite group in the following corollary.
\begin{corollary}\label{c:group_charge}
	Let $G$ be a finite group. The $\th{n}$ Gauss sum of $\ZZ(\Rep(G))$ is
	$$
	\t_n = \t_{-n}= |G|\cdot |\{x \in G\mid x^n=1\}| >0
	$$
	for all $n \in \BZ$.
\end{corollary}
\begin{proof}
	Recall that $\nu_n(\Rep(G)) = \nu_n(\BC[G]) =|\{x \in G\mid x^n=1\}|$ for any non-negative integer $n$. Since
	$\nu_{n}(\Rep(G))$ is a positive integer,  $\nu_{-n}(\Rep(G)) = \nu_{n}(\Rep(G))$.   The equation follows immediately from Corollary \ref{c:center_gauss_sum2}.
\end{proof}

\begin{example}\label{ex:zero_gauss_sum} In general, the higher Gauss sums of modular categories could be zero.  Here, we list two examples of integral modular categories of even dimension with zero second Gauss sums.
	\begin{enumerate}
		\item[(i)] If $\CC=\CC(\BZ_2, q)$ where  $q(1)=i$ for the nontrivial simple object $X$, then $\tau_2(\CC) = 0$ and so
		the second central charge is undefined despite $\mathcal{C}(\mathbb{Z}_2,q)$ being modular.
		\item[(ii)] It has been shown in \cite[Section 5]{Shi1} that there exists a semisimple Hopf algebra $H$ of even dimension with $\Tr(S_H)=0$. Therefore, if $\CC=\ZZ(\Rep(H))$, then $\tau_2(\CC)=0$ by Corollary \ref{c:center_gauss_sum2}.
	\end{enumerate}
\end{example}
Another important consequence of Proposition \ref{p:gauss_sum} is the invariance of $\nu_n(\DD)$ of Morita equivalence classes of any pseudounitary fusion category $\DD$.
\begin{corollary}
Let $\AA$ be a pseudounitary fusion category, and $\BB$ a fusion category  Morita equivalent to $\AA$, i.e., $\ZZ(\AA)$ and $\ZZ(\BB)$ are equivalent braided fusion categories. Then $\BB$ is also pseudounitary and
$$
\nu_n(\AA) = \nu_n(\BB)\quad \text{for all }n \in \BZ,
$$
where the underlying spherical structures of $\AA$ and $\BB$ are the canonical ones determined by their pseudounitarity.
\end{corollary}
\begin{proof}
  Since $\AA$ is pseudounitary, so is $\ZZ(\AA)$. Therefore, $\ZZ(\BB)$ is pseudounitary and so
  $$
  \dim(\BB)^2=\dim(\ZZ(\BB)) = \FPdim(\ZZ(\BB))=\FPdim(\BB)^2 = \FPdim(\AA)^2.
   $$
   Since  $\dim(\BB)$, $\FPdim(\BB)$ and $\FPdim(\AA)$ are positive real numbers, $\dim(\BB) = \FPdim(\BB)=\dim(\AA)$ and hence $\BB$ is pseudounitary. Assuming both $\AA$ and $\BB$ are equipped with the canonical spherical structures  determined by their pseudounitarity, $\ZZ(\AA)$ and $\ZZ(\BB)$ are equivalent modular categories by \cite[Corollary 6.2]{NSind}. Therefore, $\t_n(\ZZ(\AA))= \t_n(\ZZ(\BB))$ for all $n \in \BZ$. Since $\dim(\AA) =\dim(\BB)$, it follows from Proposition \ref{p:center_gauss_sum1} that $\nu_n(\AA) = \nu_n(\BB)$ for all $n \in \BZ$.
\end{proof}

\subsection{Anomaly and central charge}
The quotient $\tau_1(\CC)/\sqrt{\dim (\CC)}$ of a modular category $\CC$, called the \emph{central charge} of $\CC$, is a square root of $\a=\tau_1(\CC)/\tau_{-1}(\CC)$ since $\dim(\CC) =\tau_1(\CC) \tau_{-1}(\CC) = |\tau_1(\CC)|^2$. The choice of positive square root of $\dim (\CC)$ determines a square root of  $\a$, which  is natural but not particularly a canonical one.  One can easily extend these notions to higher degrees for premodular categories.
\begin{definition}\label{def:gausssum2}
	Let $ \CC $ be a premodular category. For $n\in\mathbb{Z}$, we respectively define the \emph{$ \th{n} $ anomaly} and  the \emph{$ \th{n} $ (multiplicative) central charge} of $ \CC $  as
	\begin{equation}
	\alpha_n(\CC) := \frac{\tau_n(\CC)}{\tau_{-n}(\CC)} \quad\text{and}\quad  \xi_n(\mathcal{C}):=\frac{\tau_n(\mathcal{C})}{|\tau_n(\mathcal{C})|}
	\end{equation}
	provided $\tau_n(\CC)\neq0$.
\end{definition}
The first anomaly appears as an important quantity in the 3-dimensional topological quantum field theory defined by a modular category $ \CC $ (see for example \cite{TuraevBook}).  The motivation for our definitions of higher Gauss sums and central charge is closely related to the Reshetikhin-Turaev invariants of links and 3-manifolds arising from modular categories.

Let $\mathcal{C}$ be a modular category and $D$ the positive square root of $\dim (\CC)$.  Denote the RT-invariant of a 3-manifold $M$ corresponding to $\mathcal{C}$ by $\RT(M)$. Let $n \in \BN$. By \cite[Section II.2.2]{TuraevBook}, the RT-invariant of the lens space $L(n, 1)$ associated to $\mathcal{C}$ is given by
\begin{equation}
\RT(L(n, 1))=D^{-3}\tau_{-1}\tau_n
\end{equation}
where the category $\mathcal{C}$ in the notation of the Gauss sums has been suppressed for brevity. The same manifold with reversed orientation, denoted by $-L(n, 1)$, has RT-invariant
\begin{equation}
\RT(-L(n, 1))=D^{-1}
(\tau_{-1})^{-1}\tau_{-n}.
\end{equation}
Observe that $\RT(-L(n, 1))=\overline{\RT(L(n, 1))}$. If any one of the invariants is nonzero, we have
\begin{equation}
\frac{\RT(L(n, 1))}{\RT(-L(n, 1))}=D^{-2}(\tau_{-1})^{2}\a_n=\frac{\a_n}{\a_1}.
\end{equation}
\end{subsection}


\begin{subsection}{Basic properties and examples}

Some straightforward observations about Definitions \ref{def:gausssum} and \ref{def:gausssum2} are collected here for future use.
\begin{lemma}\label{lem1}
	For all premodular categories $\mathcal{C}$ and integers $n$ such that $\tau_n(\mathcal{C})\neq0$,
	\begin{itemize}
		\item[\textnormal{(i)}] $\xi_n(\mathcal{C}\rev)=\xi_{-n}(\CC) = \ol{\xi_{n}(\CC)}$,
		\item[\textnormal{(ii)}]  $\xi_n(\mathcal{C})$ and all of its Galois conjugates have modulus 1,
		\item[\textnormal{(iii)}] $\xi_n(\mathcal{C})$ is a root of unity if and only if $\xi_n(\mathcal{C})$ is an algebraic integer. In this case, if $N$ is the FS-exponent of $\CC$, then $\ord(\a_n(\CC)) \mid N$ if $N$ is even, and $\ord(\a_n(\CC))\mid 2N$ otherwise.
	\end{itemize}
\end{lemma}

\begin{proof}
	The first  claim follows from $\tau_{\pm n} (\CC)=\tau_{\mp n}(\CC^\text{rev})$ and $\tau_n(\mathcal{C})=\overline{\tau_{-n}(\mathcal{C})}$.  Note that $\a_n(\CC) =\frac{\t_n(\CC)}{\ol{\t}_{n}(\CC)} \in\BQ(e^{2\pi i/N})$. Therefore, all the Galois conjugates of $\a_n(\mathcal{C})$ have modulus 1, and so do $\xi_n(\CC)$.  If $\xi_n(\CC)$ is an algebraic integer, then $\xi_n(\CC)$ is a root of unity by a theorem of Kronecker (see \cite{kronecker2}). Hence, $\alpha_n(\CC)$ is also a root of unity. Since $\a_n \in  \BQ(e^{2\pi i/N})$, if $\a_n(\CC)$ is a root of unity, then $\ord(\a_n(\CC))\mid N$ or $\ord(\a_n(\CC))\mid 2N$ respectively depends on  whether  $N$ is even or odd.
\end{proof}

\begin{example}
	Accessible formulas for the dimensions and twists of $\mathcal{C}(\mathfrak{g}_2,3)$ can be found in Sections 2.3.4 of \cite{schopieray2}, and one computes
	\begin{equation*}
	\xi_3(\mathcal{C})=\dfrac{1}{2\sqrt{2}}(\sqrt{7}-i),
	\end{equation*}
	which has minimal polynomial $2x^4-3x^2+2$ and thus is not a root of unity.  Corollary \ref{cor:RootOfUnity} below describes when such a phenomenon is possible.
\end{example}

The higher Gauss sums and central charges also respect the Deligne tensor product of premodular categories.
\begin{lemma}\label{lem:mult}
	If $\mathcal{C}$ and $\mathcal{D}$ are premodular, then for all $n\in\mathbb{Z}$
	\begin{equation}\tau_n(\mathcal{C}\boxtimes\mathcal{D})=\tau_n(\mathcal{C})\tau_n(\mathcal{D}).\end{equation}
	If $\t_n(\CC)\t_n(\DD)\ne 0$, then we also have $\xi_n(\mathcal{C}\boxtimes\mathcal{D})=\xi_n(\mathcal{C})\xi_n(\mathcal{D})$.
\end{lemma}

\begin{proof}
	The result follows from the fact that dimensions and twists are multiplicative with respect to $\boxtimes$, i.e. $\dim(X\boxtimes Y)=\dim_\mathcal{C}(X)\dim_\mathcal{D}(Y)$ and $\theta_{X\boxtimes Y}=\theta_X \theta_Y$ for any $X \in \OO(\CC)$ and $Y \in \OO(\DD)$. Hence
	\begin{align*}
	\tau_n(\mathcal{C}\boxtimes\mathcal{D})&=\sum_{X\boxtimes Y\in\mathcal{O}(\mathcal{C}\boxtimes\mathcal{D})}\theta_{X\boxtimes Y}^n\dim(X\boxtimes Y)^2 \\
	&=\sum_{X\in\mathcal{O}(\mathcal{C})}\sum_{Y\in\mathcal{O}(\mathcal{D})}\theta^n_X \theta^n_Y \dim_\CC(X)^2\dim_\DD(Y)^2 \\
	&=\tau_n(\mathcal{C})\tau_n(\mathcal{D}).
	\end{align*}
	The last statement follows directly from the definition of $\xi_n$.
\end{proof}
\begin{corollary}\label{cor:mod}
	If $\mathcal{C}$ is modular, then $\xi_n(\mathcal{Z}(\mathcal{C}))=1$ for all $n\in\mathbb{Z}$ such that $\tau_n(\mathcal{C})\neq0$.
\end{corollary}

\begin{proof}
	Apply Lemma \ref{lem1} (ii) to $\mathcal{Z}(\mathcal{C})\simeq\mathcal{C}\boxtimes\mathcal{C}^\text{rev}$ (see Equation (\ref{eq:center})).
\end{proof}

\begin{example}\label{wrong} One can easily see that there are families of inequivalent premodular categories which have the same higher multiplicative central charges. The first of such an example can be obtained from finite groups. Recall from Corollary \ref{c:group_charge} that for any finite group $G$ the higher Gauss sums of $\ZZ(\Rep(G))$ are positive integers. Therefore, all the higher central charges of $\ZZ(\Rep(G))$  are all equal to 1. In fact, the same property holds for $\Rep(G)$. Since $\t_n(\Rep(G))=|G|$, $\xi_n(\mathrm{Rep}(G))=1$ for all $n\in\mathbb{Z}$.
\end{example}

The modular categories in Example \ref{wrong} are all contained in the trivial Witt equivalence class $[\mathrm{Vec}]$ (Section \ref{sec:local}), but the following example illustrates that higher central charges of $\CC$ could be different from 1, and they are the first central charges of the Galois conjugates of $\CC$.

\begin{example}\label{wrong2}
	Let $p\in\BN$ be an  odd prime and $q_a: \BZ_p \to \BC^\times$ a nondegenerate quadratic form such that $q_a(1) = e^{2\pi i a/p}$ for some integer $a$ not divisible by $p$. Then the Gauss sum of $\CC_a=\CC(\BZ_p, q_a)$ is identical to the classical quadratic Gauss sum of $q_a$, which is given by
	$$
	\t_1(\CC_a) = \sum_{j=0}^{p-1} q_a(j) =
	\jac{a}{p}\sqrt{\jac{-1}{p} p}
	$$
	where $\jac{a}{p}$ is the Legendre symbol. Thus, for $p \nmid n$, the $\th{n}$ Gauss sum and multiplicative charge of $\CC_a$ are respectively
	$$
	\t_n(\CC_a) = \sum_{j=0}^{p-1} q_{an}(j) = \jac{an}{p}\sqrt{\jac{-1}{p} p}  \quad\text{and}\quad \xi_n(\CC_a) = \jac{an}{p}\sqrt{\jac{-1}{p}}\,.
	$$
	Therefore, $\t_n(\CC_a)=\t_1(\CC_{na})$ and $\xi_n(\CC_\a) = \xi_1(\CC_{na})$. Since $\CC_a$ is equivalent to $\CC_b$ if and only if $\t_1(\CC_a)= \t_1(\CC_b)$, there are only two inequivalent modular categories among $\CC_a$ for a given prime $p$ which are determined by the Legendre symbol $\jac{a}{p}$. These two inequivalent modular categories, which can be distinguished by their central charges $\xi_1$, are also generators of the Witt group $\WW_{pt}(p)$.
\end{example}

In light of Propositions \ref{p:gauss_sum} and \ref{p:center_gauss_sum1}, one can see the higher Gauss sums are closely related to the higher Frobenius-Schur indicators, and they are invariants of premodular categories. The following example is an application of the higher Gauss sums to distinguish the Drinfeld centers of the Tambara-Yamagami categories.
\begin{example}\label{ex:one}
	A Tambara-Yamagami ($\TY$-)category $\CC$ is a fusion category with $\OO(\CC) = A \sqcup \{m\}$ where $A$ is a finite abelian group and the fusion rules are given by
	$$
	m \o  m = \sum_{a \in A} a,
	\quad
	a \o m = m \o a =m,
	\quad
	a \o b = ab
	\quad\text{for } a, b \in A.
	$$
	A $\TY$-category is completely determined by the abelian group $A$, a symmetric nondegenerate bicharacter $\chi$ of $A$, and a square root $r$ of $|A|\inv$, and is denoted by $\TY(A, \chi, r)$ (cf. \cite{tamyam}).  Every $\TY$-category is pseudounitary (cf. \cite{ENO}) and its Drinfeld center is consequently a modular category. The $\TY$-categories defined by the abelian group $A=\BZ^2_2$ are representation categories of quasi-Hopf algebras and they are completely distinguished by their higher Frobenius-Schur indicators (cf. \cite{NS07b}). The group $\BZ_2^2$ admits two inequivalent nondegenerate symmetric bicharacters, namely the standard symmetric bicharacter $\chisym$ and the alternating bicharacter $\chialt$.  Using the higher FS-indicators computed in \cite{NS07b} or \cite{Shi1}, we have the following table of higher Gauss sums for the Drinfeld centers of these $\TY$-categories.
	$$
	\begin{array}{|c|c|cccccccc|}
	\hline
	\CC & H & \t_1 & \t_2 & \t_3 & \t_4 & \t_5 & \t_6 & \t_7 & \t_8 \\
	\hline
	\TY(A, \chisym,\frac{1}{2} )& \BC[D_8] & 8 & 48 & 8 & 64 & 8 & 48 & 8 & 64 \\
	\TY(A, \chialt,\frac{-1}{2} ) & \BC[Q_8] & 8 & 16 & 8 & 64 & 8 & 16 & 8 & 64 \\
	\TY(A, \chisym,\frac{1}{2} ) & K & 8 & 48 & 8 & 32 & 8 & 48 & 8 & 64\\
	\TY(A, \chialt,\frac{-1}{2} ) & K_u & 8 & 16 & 8 & 32 & 8 & 16 & 8 & 64\\
	\hline
	\end{array}
	$$
	Here $\CC$ is equivalent to $\Rep(H)$ as spherical fusion categories, $\t_n$ is the $\th{n}$ Gauss sum of $\ZZ(\CC)$, $K$ is the Kac algebra of dimension 8, and $K_u$ is a twist of $K$ defined in \cite{NS07b}. Since these four  sequences higher Gauss sums are different, the Drinfeld centers of these $\TY$-categories are inequivalent modular categories. However, all the higher central charges of these modular categories are 1.
\end{example}

The repetition of higher central charges is also revealed by the following proposition.
\begin{proposition}\label{th:ferst}
	Let $\mathcal{C}$ be a premodular category.  Then $\t_1(\CC) \t_{-2}(\CC) \in \BR$. In particular, if $\t_1(\CC) \t_{-2}(\CC) \ne 0$, then $\a_1(\CC)=\a_2(\CC)$ or $\xi_1(\CC)=\pm \xi_2(\CC)$.
\end{proposition}

\begin{proof}
	Using the graphical calculus conventions of \cite{BakalovKirillov}, we have
	\begin{equation}\label{eq:3}
	\tau_1(\mathcal{C})\tau_{-2}(\mathcal{C})=\tau_1(\mathcal{C})
	\begin{tikzpicture}[baseline=0]
	\node (a) at (0,.1) {};
	\node[left = 0.1cm of a] (b) {};
	\node[right = 0.5cm of b,circle,draw] (c) {$\theta^{-2}$};
	\node[above = 0.75cm of b] (b2) {};
	\node[above = 0.4cm of c] (c2) {};
	\node[below = 0.75cm of b] (b3) {};
	\node[below = 0.4cm of c] (c3) {};
	\draw (b2.center) edge (b.center);
	\draw (b3.center) edge (b.center);
	\draw (c2.center) edge (c);
	\draw (c3.center) edge (c);
	\draw[in=90,out=90] (b2.center) edge (c2.center);
	\draw[in=-90,out=-90] (b3.center) edge (c3.center);
	\end{tikzpicture}
	=\tau_1(\mathcal{C})\,\,
	\begin{tikzpicture}[baseline=0]
	\node (a) at (0,.1) {};
	\node[left = 0cm of a,circle,draw] (b) {$\theta^{-1}$};
	\node[right = 0.5cm of b,circle,draw] (c) {$\theta^{-1}$};
	\node[above = 0.75cm of b] (b2) {};
	\node[above = 0.75cm of c] (c2) {};
	\node[below = 0.75cm of b] (b3) {};
	\node[below = 0.75cm of c] (c3) {};
	\draw (b2.center) edge (b);
	\draw (b3.center) edge (b);
	\draw (c2.center) edge (c);
	\draw (c3.center) edge (c);
	\draw[in=90,out=90] (b2.center) edge (c2.center);
	\draw[in=-90,out=-90] (b3.center) edge (c3.center);
	\end{tikzpicture}
	\end{equation}
	where the line is labeled by the pseudo-object $\sum_{X \in \OO(\CC)} \dim_\CC(X) X$.
	Note that \cite[Lemma 3.1.5]{BakalovKirillov} holds for premodular categories by the same proof. Applying \cite[Lemma 3.1.5]{BakalovKirillov} to the last term of \eqref{eq:3}, we have
	\begin{equation}\label{eq:4}
	\tau_1(\mathcal{C})\tau_{-2}(\mathcal{C})
	=
	\OnePlusTwoMinus
	=
	\StretchCrossing,
	\end{equation}
	where the second equality follows from the rigidity of $\CC$. Since $\CC$ is spherical, we have
	\begin{equation}
	\tau_1(\mathcal{C})\tau_{-2}(\mathcal{C})
	=
	\OneMinusTwoPlus
	=
	\tau_{-1}(\mathcal{C})\tau_2(\mathcal{C}),
	\end{equation}
	where the last equality follows from  \cite[Lemma 3.1.5]{BakalovKirillov}. The second assertion follows directly from the definition.
\end{proof}

\section{Arithmetic properties of higher Gauss sums: modular case}\label{sec:arith}

In this section, we study the action of the Galois group of $\Qb$ on the higher Gauss sums of a modular category $\CC$. We obtain a relation of the higher Gauss sums in terms of the action of automorphisms of $\Qb$ in Theorem \ref{thm:Galois}. In particular, we prove in Corollary \ref{cor:RootOfUnity} that those Gauss sums $\t_n(\CC)$ with $n$ relatively prime to $\ord(T_\CC)$ are nonzero $d$-numbers, and the corresponding central charges are roots of unity. We also extend a result of M\"uger on the first Gauss sum of the Drinfeld center of a spherical fusion category to higher Gauss sums in Theorem \ref{thm:GaussSumOfCenter}.

Let $\mathcal{C}$ be a modular category with the unnormalized S- and T-matrices $S$ and $T$ respectively. Set $s:=\frac{1}{\sqrt{\dim(\CC)}} S$. Then, $s^2_{X,\1}=\dim(X)^2/\dim(\CC)$.  For any $n\in\BN$, denote $ \zeta_n:=e^{2\pi i/n}$.  It was quite well known (see \cite{deBoerGoeree, CosteGannon94, ENO}) that all entries of $ S $ and $ T $ are cyclotomic integers.
It has been recently shown (see \cite[Proposition 5.7]{NS10})  that  they are elements of $\Q(\zeta_{\ord(T)}) $. Thus, the Galois group $\Gal(\Q(\zeta_{\ord(T)})/\Q)$ acts on the modular data of $\CC$. The reader is referred to \cite{dong2015} or \cite{CosteGannon94} for more details of Galois group actions on the modular data.

For any automorphism $\sigma$ of  $\Qb$, there exists a unique permutation $ \hat{\sigma} $ on $\OO(\cat{C})$ such that
\begin{equation}
\sigma\left(\frac{s_{X,Y}}{s_{\1,Y}}\right)
=
\frac{s_{X,\hat{\sigma}(Y)}}{s_{\1 ,\hat{\sigma}(Y)}}
\end{equation}
for all $ X, Y \in \OO(\cat{C}) $. Moreover, $\sigma$ gives rise to a sign function $\e_\sigma: \mathcal{O}(\mathcal{C})\to\{\pm1\}$ such that
\begin{equation}
\sigma(s_{X,Y})
=
\epsilon_\sigma(X)s_{\hat{\sigma}(X),Y}
=
\epsilon_\sigma(Y)s_{X,\hat{\sigma}(Y)}
\end{equation}
for all $X,Y \in \OO(\cat{C})$. In particular, for any $ X \in \OO(\cat{C}) $, we have
\begin{equation}\label{eq:GaloisOnNormalizedDimension}
\sigma(s_{X, \1}^2)
=
s_{\hat{\sigma}(X),\1}^2.
\end{equation}
Fix any $ \gamma \in \BC $ such that
\begin{equation}\label{eq:NormalizationFactor}
\gamma^3
=
\xi_1(\cat{C})
=
\frac{\tau_1(\cat{C})}{\sqrt{\dim(\CC)}},
\end{equation}
define $t := \gamma^{-1}T$.  Then the assignment
\begin{equation}
\rho:
\begin{bmatrix}
0 & -1\\
1 & 0
\end{bmatrix}
\mapsto s,
\
\begin{bmatrix}
1 & 1\\
0 & 1
\end{bmatrix}
\mapsto t
\end{equation}
uniquely determines a linear representation of $\SL(2, \BZ) $ (cf. \cite[Remark 3.1.9]{BakalovKirillov}).

Note that $t$ is a diagonal matrix with entries indexed by $X\in\mathcal{O}(\mathcal{C})$. We will denote the diagonal entry of $ t $ corresponding to an $ X \in \OO(\CC) $ by $t_X $. In other words, $t_X=\theta_X\gamma^{-1}$ for any $X\in\OO(\CC)$. According to \cite[Theorem II]{dong2015},
\begin{equation}\label{eq:GaloisOnT}
\sigma^2(t_{X}) = t_{\hat{\sigma}(X)}.
\end{equation}
Because $\theta_{\1} = 1 $, Equation (\ref{eq:GaloisOnT}) states
\begin{equation}
\sigma^2\left(\frac{1}{\gamma}\right) = \frac{\theta_{\hat{\sigma}(\1)}}{\gamma},
\end{equation}
which implies \cite[Proposition 4.7]{dong2015}
\begin{equation}\label{eq:GammaAndTheta}
\frac{\gamma}{\sigma^2(\gamma)} = \theta_{\hat{\sigma}(\1)}.
\end{equation}

For any $ n \in \mathbb{Z}$ relatively prime to $\ord(T)$, there exists an automorphism $\sigma_n$ of $\Qb$ such that $\sigma_n: \zeta_{\ord(T)}\mapsto\zeta_{\ord(T)}^n$ (cf. \cite[Theorem VI.3.1]{LangAlgebra}). In particular, $ \sigma_n(\theta_X) = \theta_X^n $ for all $ X \in \OO(\CC) $. Let $ \tilde{n} $ denote the multiplicative inverse of $n$ modulo $\mathrm{ord}(T)$.  Then
\begin{equation}\label{eq:MultiplicativeInverse}
\sigma^{-1}_{\tilde{n}}(\theta_X) = \sigma_{n}(\theta_X) = \theta_X^{n}
\end{equation}
for all $ X \in \OO(\CC)$. Note that such $\sigma_n$ is not unique but its restriction on $\BQ(\zeta_{\ord(T)})$ is unique. Now we can state our theorem of Galois group action on the higher Gauss sums.

\begin{theorem}\label{thm:Galois}
	Let $\mathcal{C}$ be a modular category, $N=\ord(T_\CC)$, and  $n\in\BZ$ relatively prime to $N$. Then, for $a \in \BZ$,
	\begin{equation}
	\tau_{an}(\CC)=\sigma(\tau_a(\CC))\frac{\dim(\CC)}{\sigma(\dim(\CC))}\theta_{\hat{\sigma}(\1)}^{an} \quad and\label{eq:TauN}
	\end{equation}
\begin{equation}
	\nu_{an}(\CC)=\sigma(\nu_a(\CC))\frac{\dim(\CC)}{\sigma(\dim(\CC))}, \label{eq:NuN}
	\end{equation}
	where $\tilde{n}$ is the multiplicative inverse of $n$ modulo $N$, and $\sigma$ is any automorphism  of $\Qb$ such that $\sigma(\zeta_N)=\zeta_N^{\tilde{n}}$. In particular,  $\tau_n(\CC)\neq 0$ and $\xi_n(\CC)$ is well-defined.
\end{theorem}

\begin{proof}
	Fix a $3^{rd}$ root $\gamma$ of $ \xi_1(\CC)$ (cf. Equation (\ref{eq:NormalizationFactor})).  By Equations (\ref{eq:GaloisOnNormalizedDimension}), (\ref{eq:GaloisOnT}) and (\ref{eq:MultiplicativeInverse}), we have
	\begin{equation*}
	\begin{aligned}
	\sigma
	\left(
	\frac{\tau_a(\CC)}{\dim(\CC)\gamma^a}
	\right)
	&=
	\sigma
	\left(
	\sum_{X \in \OO(\CC)}
	s^2_{X,\1} t^a_X
	\right)\\
	&=
	\sum_{X \in \OO(\CC)}
	\sigma(s^2_{X,\1})
	\sigma(t^a_X)\\
	&=
	\sum_{X \in \OO(\CC)}
	s^2_{\hat{\sigma}(X),\1}
	\sigma^{-1}(\sigma^2(t^a_X)) \quad\text{by \eqref{eq:GaloisOnNormalizedDimension}}\\
	&=
	\sum_{X \in \OO(\CC)}
	s^2_{\hat{\sigma}(X),\1}
	\sigma^{-1}(t^a_{\hat\sigma(X)}) \quad \text{ by \eqref{eq:GaloisOnT}} \\
	&=
	\sum_{X \in \OO(\CC)}
	\frac{\dim(\hat{\sigma}(X))^2 \sigma^{-1}(\theta^a_{\hat{\sigma}(X)})}
	{\dim(\CC)\sigma^{-1}(\gamma^a)}\\
	&=
	\frac{1}
	{\dim(\CC)\sigma^{-1}(\gamma^a)}
	\sum_{X \in \OO(\CC)}
	\dim(\hat{\sigma}(X))^2 \theta^{an}_{\hat{\sigma}(X)}
	\quad
	\text{by \eqref{eq:MultiplicativeInverse}}
	\\
	&=
	\frac{\tau_{an}(\CC)}{\dim(\CC) \sigma^{-1}(\gamma^a)},
	\end{aligned}
	\end{equation*}
	where the last equality is based on the fact that $ \hat\sigma $ is a permutation on $ \OO(\CC) $. Therefore, with $M:=\sigma(\tau_a(\CC))\dim(\CC)/\sigma(\dim(\CC))$ for brevity,	
	\begin{equation}
	\tau_{an}(\CC)= M \frac{\sigma^{-1}(\gamma^a)}{\sigma(\gamma^a)}
=M\sigma^{-1}\left(\frac{\gamma^a}{\sigma^2(\gamma^a)}\right)
=M\sigma^{-1}(\theta^a_{\hat{\sigma}(\1)})
=M\theta^{an}_{\hat{\sigma}(\1)}
	\end{equation}
	by Equations (\ref{eq:GammaAndTheta}) and (\ref{eq:MultiplicativeInverse}). This proves \eqref{eq:TauN}. The equality \eqref{eq:NuN} is then a consequence of Proposition \ref{p:gauss_sum}.  The last statement follows immediately from the fact that $|\t_1(\CC)|^2=\dim (\CC) >1$ as $\CC$ is modular.
\end{proof}

Recall from \cite{codegrees} that a \emph{$d$-number} is defined as an algebraic integer whose principal ideal in the ring of algebraic integers is invariant under the action of the absolute Galois group. It is immediate from the definition that the subset of $d$-numbers in the ring of algebraic integers is closed under multiplication and taking square roots, and that any algebraic unit is a $d$-number. Moreover it is shown in \cite[Corollary 1.4]{codegrees} that if $ \CC $ is a spherical fusion category, then $ \dim(\CC) $ is a $d$-number.

\begin{corollary}\label{cor:RootOfUnity}
	Let $ \CC $ be a modular  category,   $N=\ord(T_\CC)$, $n\in\mathbb{Z}$ with $\gcd(n,N)=1$ and $\tilde{n}$ a multiplicative inverse of $n$ modulo $N$, and $\sigma$ is any automorphism  of $\Qb$ such that $\sigma(\zeta_N)=\zeta_N^{\tilde{n}}$. Then,
	\begin{itemize}
		\item[\textnormal{(a)}] the $\th{n}$ anomaly of $\CC$ is given by \begin{equation}
		\alpha_n(\CC) = \sigma(\a_1(\CC))\cdot\theta^{2n}_{\hat{\sigma}(\1)}\,.
		\end{equation}
		In particular,  $\alpha_n(\CC) $ and $ \xi_n(\CC) $ are both roots of unity such that $ \xi_n(\CC)^{4N} = 1$.
		\item[\textnormal{(b)}] The $\th{n}$ Gauss sum of $\CC$ is a $d$-number.
	\end{itemize}
\end{corollary}

\begin{proof}
	By Theorem \ref{thm:Galois},  we have
	\begin{equation}
	\begin{aligned}
	\alpha_n(\CC)
	=
	\frac{\tau_n(\CC)}{\tau_{-n}(\CC)}
	=
	\frac
	{\sigma(\tau_1(\CC)) \theta^{n}_{\hat{\sigma}(\1)}}
	{\sigma(\tau_{-1}(\CC)) \theta^{-n}_{\hat{\sigma}(\1)}}
	=
	\sigma(\alpha_1(\CC)) \cdot \theta^{2n}_{\hat{\sigma}(\1)}.
	\end{aligned}
	\end{equation}
	It is known (for example, \cite[Section 3]{BakalovKirillov}) that $ \alpha_1 \in \Q(\zeta_{N})$ is a root of unity. Therefore, $ \alpha_n(\CC) $ is a root of unity in $ \Q(\zeta_{N}) $. As a square root of $ \alpha_n(\CC) $, $ \xi_n(\CC) $ is also a root of unity, this completes the proof of part (a).
	
	\par Now, by Proposition \ref{p:gauss_sum}, we have
	\begin{equation}
	\tau_n(\CC)^2 = |\tau_n(\CC)|^2\a_{n}(\CC) =\dim(\CC)\nu_n(\CC)\alpha_n(\CC).
	\end{equation}
	By Proposition 3.6 and Proposition 3.8 of \cite{paul}, $ \nu_q(\CC) $ is an algebraic unit for all prime $q \nmid N$. By the Dirichlet prime number theorem, there exists a prime number $q$ such that $q \equiv n \pmod N$. Hence, $\nu_q(\CC)=\nu_n(\CC)$ and so $ \nu_n(\CC) $ is an algebraic unit.  Since $\dim(\CC)$ is a $d$-number, $\t_n^2(\CC)$ is a $d$-number and so are its square roots.
\end{proof}

For the Drinfeld center of a spherical fusion category, we have more explicit descriptions of its Gauss sums and anomalies. We first obtain the twists of the Galois orbit of $\1$.
\begin{lemma}\label{l:galois_twists}
Let $\DD$ be a spherical fusion category, and $\theta$ the ribbon isomorphism of $\ZZ(\DD)$. Then, for any automorphism $\s$ of $\Qb$,
$$
\theta_{\hs(\1)} = 1\,.
$$
\end{lemma}
\begin{proof}
 By \cite[Theorem 1.2]{MugerSubfactor2},  $\tau_1(\ZZ(\DD))=\tau_{-1}(\ZZ(\DD))=\dim(\DD)$, and so $\xi_1(\ZZ(\DD))=1$.  By \cite[Proposition 4.7]{dong2015} or \eqref{eq:GammaAndTheta}, we find
	$\theta_{\hat{\sigma}(\1)} = 1$ since 1 is a $3^{rd}$ root of $\xi_1(\ZZ(\DD))$.
\end{proof}
Now, we can prove our second main theorem of this section, which  extends a result of M\"uger \cite[Theorem 1.2]{MugerSubfactor2} to higher Gauss sums.
\begin{theorem}\label{thm:GaussSumOfCenter}
	Let $ \DD $ be a spherical fusion category, $N=\ord(T_{\ZZ(\DD)} )$,   $n\in\mathbb{Z}$ with $\gcd(n,N)=1$,  $\tilde{n}$ a multiplicative inverse of $n$ modulo $N$, and $\sigma$ is any automorphism  of $\Qb$ such that $\sigma(\zeta_N)=\zeta_N^{\tilde{n}}$.  Then, for $a \in \BZ$,
\begin{equation}\label{eq:galois_tau}
	\quad\tau_{an}(\ZZ(\DD))= \sigma(\t_a(\ZZ(\DD)))\frac{\dim(\DD)^2}{\sigma(\dim(\DD))^2}\quad \text{and}
	\end{equation}
	\begin{equation}\label{eq:galois_nu}
\nu_{an}(\DD)= \sigma(\nu_a(\DD))\frac{\dim(\DD)}{\sigma(\dim(\DD))}.
	\end{equation}
Moreover,
\begin{equation}\label{eq:galois_charge}
\xi_{an}(\ZZ(\DD)) =\frac{\sigma(\nu_{a}(\DD))}{|\sigma(\nu_{a}(\DD))|}
\end{equation}
whenever $\nu_a(\DD) \ne 0$. In particular, $\xi_{an}(\ZZ(\DD))=1$ whenever $\sigma(\nu_{a}(\DD))$ is positive.  For $a =1$, we always have
\begin{equation} \label{eq:galois_charge_tau}
	\xi_n(\ZZ(\DD))=1 \quad\text{and}\quad\tau_n(\ZZ(\DD))=\tau_{-n}(\ZZ(\DD))=\frac{\dim(\DD)^2}{\sigma(\dim(\DD))}.
	\end{equation}
\end{theorem}
\begin{proof}
	By Theorem \ref{thm:Galois} and Lemma \ref{l:galois_twists}, we have
	\begin{equation*}
	\tau_{an}(\ZZ(\DD))
	=
	\s(\tau_{a}(\ZZ(\DD)))
	\frac{\dim(\ZZ(\DD))}{\sigma(\dim(\ZZ(\DD))}
	=	
	\s(\tau_{a}(\ZZ(\DD))) \frac{\dim(\DD)^2}{\sigma(\dim(\DD))^2}.
	\end{equation*}
	Now, Proposition \ref{p:center_gauss_sum1} implies Equations \eqref{eq:galois_nu} and \eqref{eq:galois_charge} whenever $\nu_a(\DD) \ne 0$ since $\dim(\DD)$ is totally positive. Thus, if $\sigma(\nu_a(\DD)) >0$, $\xi_{an}(\ZZ(\DD))=1$.
By \cite[Theorem 1.2]{MugerSubfactor2}, $\t_1(\ZZ(\DD))=\t_{-1}(\ZZ(\DD))=\dim(\DD)$. Therefore, the second equation of \eqref{eq:galois_charge_tau} follows directly from \eqref{eq:galois_tau}. Now, we find $\t_n(\ZZ(\DD))$ is also totally positive, and so $\xi_n(\ZZ(\DD))=1$.
\end{proof}

It has been shown in Example \ref{wrong} that the higher central charges of Tannakian categories and their Drinfeld centers are all 1. Using the same notations as in the above theorem, the following examples of non-integral modular categories indicate that the case when $ n $ is not relatively prime to $ N $ is more subtle, even when the higher central charges are well-defined.

\begin{example}
Let $ \mathcal{H} $ be Haagerup-Izumi fusion category of rank 6 \cite{EvansGannon11}. The unitary fusion category $\HH$ is tensor generated by one simple object $\rho$, and $\OO(\HH)=G \sqcup G\rho$ where $G$ is a multiplicative group of order 3.  The fusion rules of $ \HH $ are given by the group multiplication of $G$ together with
$$
a \o \rho = a \rho,\quad a \otimes b\rho = (ab) \otimes \rho = b\rho \otimes a^{-1} \quad\text{and}\quad
a\rho \otimes b\rho = ab^{-1} \oplus \bigoplus_{c \in G} c\rho
$$
for all $ a, b \in G$. By unitarity, $ \dim(a\rho) = \frac{3+\sqrt{13}}{2} $ for any $ a \in G $. The Drinfeld center $\ZZ(\HH)$ has rank 12, and its modular data can be found in \cite{EvansGannon11}, which in principle enables us to compute all the central charges. One particularly useful information given by the modular data is that $ \ord(T_{\ZZ(\mathcal{H})}) = 39$. Therefore, by Theorem \ref{thm:GaussSumOfCenter}, if $n$ is not a multiple of 3 or 13, then $ \xi_n(\ZZ(\mathcal{H})) = 1$.

For any multiple $n$ of $3$ or $13$, we compute $\xi_n(\ZZ(\mathcal{H}))$ by using $\nu_n(\HH)$. Note that in $ \BZ_{39} $, every multiple of 3 can be written as $ 3k $ for some $ k $ such that $ \gcd(k, 39) = 1 $ (for example, $ 9 \equiv 48 = 3 \times 16$ mod 39), and the same is true for multiples of 13. Therefore, by Theorem \ref{thm:GaussSumOfCenter}, we only have to compute $ \nu_3(\HH) $ and $ \nu_{13}(\HH) $, and use Galois actions to get the indicators at other multiples of 3 and 13.

By \cite[Theorem 5.4]{TuckerNGHI}, $ \nu_3(a\rho) = 1 $, and $ \nu_{13}(a\rho) = \frac{1 + \sqrt{13}}{2}$ for any $ a \in G$. Since $G$ generates a fusion subcategory equivalent to $\Rep(G)$ as spherical fusion categories, $\nu_3(\Rep(G))=3$ and $\nu_{13}(\Rep(G))=1$.  Therefore,
$$
\nu_3(\HH)
=
\nu_3(\Rep(G)) + \sum_{a\in G} \nu_3(a\rho) \dim(a\rho)
=
3 + 3(\frac{3+\sqrt{13}}{2})
=
\frac{15+3\sqrt{13}}{2}
$$
and similarly
$$
\nu_{13}(\HH)
=
1 + 3(\frac{3+\sqrt{13}}{2})(\frac{1+\sqrt{13}}{2})
=
13+3\sqrt{13}.
$$
Consequently, by Proposition \ref{p:center_gauss_sum1}, $ \tau_3(\ZZ(\mathcal{H})) $ and $ \tau_{13}(\ZZ(\mathcal{H})) $ are positive real numbers, hence $ \xi_3(\ZZ(\mathcal{H})) = \xi_{13}(\ZZ(\mathcal{H})) = 1 $. Note that $ \nu_3 $ and $ \nu_{13} $ are totally positive. By Theorem \ref{thm:GaussSumOfCenter}, if $ n $ is a multiple of 3 or 13, then $ \xi_n(\ZZ(\HH))=1 $. In summary, $ \xi_m(\ZZ(\HH)) = 1$ for any integer $ m $.
\end{example}

However, there are Drinfeld centers of spherical fusion categories whose $ \xi_n $ is not equal to 1 when $ n $ is not relatively prime to $ N $. In fact, as we will see in the following example, $ \xi_n $ is not even a root of unity.

\begin{example} \label{ex:shi}
Let $ H $ be the 27-dimensional Hopf algebra $ H_{27}(1, \zeta_3) $ in \cite[Table 1]{ShimizuComputation} where $ \zeta_3 $ is a primitive $ 3^{rd} $ root of unity. Let $ \CC = \ZZ(\Rep(H)) $. It is shown in loc. cit. that $ \nu_3(H) = 3(5+4\zeta_3^2) $. Therefore, by Corollary \ref{c:center_gauss_sum2}, we have
$$
\tau_3(\CC) = \dim(H)\nu_3(H) = 81(5+4\zeta_3^2).
$$
Since the minimal polynomial of
$$
\alpha_3(\CC) = \frac{81(5+4\zeta_3^2)}{81(5+4\zeta_3)}
$$
is $ 7x^2+2x+7 $, $ \alpha_3 $ is not an algebraic integer. Hence, $ \xi_3 $ cannot be a root of unity. Note that $ \dim(\CC) = \dim(\Rep(H))^2 = 3^6$. Therefore, by the Cauchy Theorem \cite{KSZ, NgSchaunburgSpherical, paul}, $ \ord(T_{\CC}) $ is also a power of 3.
\end{example}

\end{subsection}

\end{section}



\begin{section}{De-equivariantization and local modules}\label{sec:deq}

Let $\mathcal{C}$ be a premodular category with the ribbon isomorphism $\theta$, and let $A\in\mathcal{C}$ be a ribbon algebra of $\CC$, i.e. a connected \'etale algebra with $\dim_\CC(A)\ne 0$ and $\theta_A=\id_{A}$ (cf. Definition \ref{d:ribbon_alg}). In the language of \cite{KiO}\cite[Remark 3.4]{DMNO}, $A$ is a \emph{rigid $\CC$-algebra}. In this section, we will derive expressions for the higher Gauss sums and central charges of the local $A$-module category $\CC_A^0$ in terms of those of $\CC$ in two different settings (Theorems \ref{th:deeq} and \ref{th:deek}).

Recall the notations and concepts related to $A$-modules in $\mathcal{C}$ from Section \ref{sec:local}.  We will use the same convention of graphical calculus as in \cite{KiO} to prove the following lemmas which are essential to our main results of this section. Note that Lemmas \ref{lemma:trace} and \ref{prop:TraceOfTheta} are in the spirit of \cite[Theorem 2.5]{ost15} and its proof.

\begin{lemma}\label{lemma:trace}
Let $\CC$ be a premodular category and $A$ a ribbon algebra of $\CC$. Then, for any $M \in\CC_{A}$ and $f\in \End_{\CC}(M)$,
\begin{equation}
    \InnerAction=\dim_\CC(A)\Tr_{\CC}(f).
\end{equation}
\end{lemma}

\begin{proof}
By the rigidity of $A$ \cite[Figure 10]{KiO}, we have

\begin{equation}
    \InnerAction
    =
    \PassOver\ .
\end{equation}
By the naturality of the braiding of $\CC$ and the associativity of the $A$-action on $M$, we have
\begin{equation}
    \PassOver=\OuterAction=\SmallLoop.
\end{equation}
Now \cite[Lemma 1.14]{KiO} and the axioms of an $A$-module imply
\begin{equation}
    \SmallLoop=\dim_\CC(A)\SmallDot=\dim_\CC(A)\TraceOfF=
    \dim_\CC(A)\Tr_{\CC}(f).
\end{equation}
\end{proof}

\par For any $M\in\OO(\CC_A)$, define
\begin{equation}\label{eq:P}
    P_M :=
    \frac{1}{\dim_\CC(A)}
    (
    \rho_M
    c_{M, A}
    c_{A, M}
    (\id_A \otimes  \rho_M)
    (i_A \otimes \id_M)
    )
    \in
    \End_{\CC}(M),
\end{equation}
where $i_A: \1 \to A \otimes A$ is as in \cite[Definition 1.11]{KiO}, $c$ is the brading of $\CC$ and $\rho_M: A \o M \to M$ is the $A$-module structure of $M$. Note that $P_M$ is the same as $P_\pi$ in \cite[Lemma 4.3]{KiO} (with $M=X_\pi$) where $\CC$ is assumed to be modular. However, the same proof of \cite[Lemma 4.3]{KiO} can be used for a premodular category $\CC$, and so we have
\begin{equation}\label{eq:PM}
    P_M=
    \left\{\begin{array}{lcr}
    \id_M & \mbox{if }  M \in \OO(\CC_A^0)\,, \\
    0 &  \mbox{otherwise\,. }\end{array}\right.
\end{equation}

\begin{lemma}\label{prop:TraceOfTheta}
Let $\CC$ be a premodular category with the ribbon isomorphism $\theta$ and $A$ a ribbon algebra of $\CC$. Then,
\begin{itemize}
\item[\textnormal{(a)}] for any $M \in \OO(\CC_A) \setminus \OO(\CC_A^0)$, $\Tr_{\CC}(\theta_M) = 0$, and
\item[\textnormal{(b)}] for any $M \in \OO(\CC_A^0)$ and $n \in \BZ$, $\theta^n_M = \lambda^n \id_M$, where $\lambda =\theta_X$ for any simple subobject $X$ of $M$ in $\CC$. In particular, $\Tr_{\CC}(\theta^n_M) = \lambda^n\dim_A(M)\dim_\CC(A)$.
\end{itemize}
\end{lemma}

\begin{proof}
It suffices to show $\Tr_{\CC}(\theta^{-1}_M)=0$ for the statement (a) since
\begin{equation}
     \Tr_{\CC}(\theta^{-1}_M)=
     \sum_{X \in \OO(\CC)} \theta^{-1}_X \dim_{\CC}(X)[M:X]_{\CC}=
     \overline{\Tr_{\CC}(\theta_M)},
\end{equation}
where the second equality follows from $\theta_X^{-1}=\overline{\theta_X}$ and $\dim_{\CC}(X)\in\mathbb{R}$.  By \cite[Figure 16]{KiO} and Lemma \ref{lemma:trace}, we have
\begin{equation}
    \Tr_{\CC}(\theta_M^{-1}P_M)=
    \frac{1}{\dim_\CC(A)}\,\,\,\InnerActionTheta=
    \Tr_{\CC}(\theta_M^{-1}).
\end{equation}
Therefore, by Equation (\ref{eq:PM}), if $M \in \OO(\CC_A)\setminus\OO(\CC_A^0)$, then $0=\Tr_{\CC}(\theta_M^{-1}P_M)=\Tr_{\CC}(\theta_M^{-1})$.

\par For (b), let $\vartheta$ be the twist of $\CC_A^0$ and $M \in \mathcal{O}(\CC_A^0)$. Then $\vartheta_M = \lambda\id_M$ for some $\lambda\in \BC^\times$. By \cite[Theorem 1.7]{KiO},  the twist in $\CC_A^0$ is inherited from that of $\CC$. We have $\vartheta_M= \theta_M$. Therefore,  $\lambda= \theta_X$  for any $X \in \OO(\CC)$ such that $[M:X]_{\CC} \neq 0$. Thus, by \cite[Theorem 1.18]{KiO}, we have

\begin{equation}
\begin{aligned}
    \Tr_{\CC}(\theta^n_M)
    &=
    \sum_{X \in \OO(\CC)}
    \theta^n_X[M:X]_{\CC}\dim_{\CC}(X)\\
    &=
    \lambda^n
    \sum_{X \in \OO(\CC)}
    [M:X]_{\CC}\dim_{\CC}(X)\\
    &=
   \lambda^n \dim_{\CC}(M)= \lambda^n\dim_A(M)\dim_\CC(A). \qedhere
\end{aligned}
\end{equation}
\end{proof}

\begin{lemma}\label{lem:TotallyPositive}
    Let $\CC$ be a premodular category with ribbon isomorphism $\theta$. If $A\in\mathcal{C}$ is a ribbon algebra of $\CC$, then
    \begin{enumerate}
    \item[{\rm (a)}]   $\displaystyle \tau_n(\CC)=\sum_{M \in \OO(\CC_{A})}
    \dim_{A}(M) \Tr_{\CC}(\theta^n_M)$ for all $n \in \BZ$,
    \item[{\rm (b)}] $\tau_1(\CC_A^0)=\dfrac{\tau_1(\CC)}{\dim_\CC(A)}$,
    \item[{\rm (c)}] $\dim(\CC_A) =  \dfrac{\dim(\CC)}{\dim_\CC(A)}$ and  $\dim_\CC(A)$ is totally positive.
    \end{enumerate}
\end{lemma}

\begin{proof}
 By \cite[Theorem 1.18]{KiO},  $\dim_A(M) = \frac{\dim_\CC(M)}{\dim_\CC(A)}$ for any object $M \in \CC_A$. Since the forgetful functor $\CC_A \to \CC$ is a left adjoint of $F: \CC \to \CC_A$; $X \to A \o X$, we have
\begin{equation*}\label{eq:TauAndTrace}
\begin{aligned}
    \tau_n(\CC)&=
    \sum_{X\in\mathcal{O}(\mathcal{C})}
    \theta^n_X \dim_{\CC}(X)^2=
    \sum_{X\in\mathcal{O}(\mathcal{C})}
    \theta^n_X \dim_{\CC}(X) \dim_{A}(A \otimes X)\\
    &=
    \sum_{X\in\mathcal{O}(\mathcal{C})}
    \theta^n_X \dim_{\CC}(X)
    \sum_{M \in \OO(\CC_{A})}
    [A \otimes X: M]_{\CC_{A}}\dim_{A}(M)\\
    &=
    \sum_{M \in \OO(\CC_{A})}
    \dim_{A}(M)
    \sum_{X \in \OO(\CC)}\theta^n_X
    [M:X]_{\CC}\dim_{\CC}(X)\\
    &=
    \sum_{M \in \OO(\CC_{A})}
    \dim_{A}(M) \Tr_{\CC}(\theta^n_M),
\end{aligned}
\end{equation*}
and this proves part (a).

For the statement (b), we consider $n=1$. Then, Lemma \ref{prop:TraceOfTheta} implies
\begin{equation*}\label{eq:TauOneCAndTauOneCA0}
\begin{aligned}
    \tau_1(\CC)
    &=
    \sum_{M \in \OO(\CC_{A}^0)}
    \dim_{A}(M) \Tr_{\CC}(\theta_M)+
    \sum_{M \in \OO(\CC_{A})\setminus \OO(\CC_{A}^{0})}
    \dim_{A}(M)\Tr_{\CC}(\theta_M)\\
    &=
    \dim_\CC(A)\sum_{M \in \OO(\CC_{A}^0)}
    \theta_M \dim_{A}(M)^2+0\\
    &=\dim_\CC(A)\tau_1(\CC_A^0).
\end{aligned}
\end{equation*}

Now, we consider $n=0$ for equation of part (a).  We have
$$
\begin{aligned}
    \dim(\CC) = \tau_0(\CC) & = \sum_{M \in \OO(\CC_A)} \dim_A(M) \dim_\CC(M)\\
              & =\dim_\CC(A)\sum_{M \in \OO(\CC_A)} \dim_A(M)^2 =\dim_\CC(A)\dim(\CC_A)\,.
\end{aligned}
$$
Since the global dimensions $\dim(\CC)$ and $\dim(\CC_A)$ are totally positive, so is $\dim_\CC(A)$, and the proof is completed.
\end{proof}
The following corollary, which is a generalization of \cite[Theorem 4.5]{KiO} to premodular categories, is now an immediate consequence of the preceding lemma.
\begin{corollary}\label{lem:known}
Let $\CC$ be a premodular category with $\t_1(\CC)\ne 0$. If $A\in\mathcal{C}$ is a ribbon algebra of $\CC$, then
$\xi_1(\CC_A^0)=\xi_1(\CC)$. \hfill{$\square$}
\end{corollary}

Let $\CC$ be a premodular category. Recall that a fusion full subcategory $\AA$ of $\CC$ is a \emph{Tannakian subcategory} if $\AA$ is equivalent $\Rep(G)$ for some finite group $G$, as premodular categories. Let $ A = \operatorname{Fun}(G)$ be the regular algebra (the algebra of complex valued functions on $G$). In this case, $ A $ is a ribbon algebra of $\CC$ with dimension $ \dim_\CC(A) = |G| $, and $ \CC_A $ is the de-equivariantization $\CC_G$ of $ \CC $ (cf. Definition \ref{d:ribbon_alg} and \cite{DGNO}). The corresponding category of local modules is denoted by $ \CC_G^0 (=\CC_A^0) $.

Let $G$ be a finite subgroup of the group of the isomorphism classes of invertible objects of $\CC$. We simply call $G$ a \emph{Tannakian subgroup} of $\CC$ if the full subcategory $\AA$ of $\CC$ generated by $G$ is a Tannakian subcategory of $\CC$. In this case, $ G $ is abelian, $ \OO(\AA) = G $ and $\AA$ is equivalent to the premodular category $\Rep(\hat{G})$, where $\hat{G}\, (\cong G)$ is the character group of $G$. Therefore, the regular algebra $A$ of $\AA$ is given by
\begin{equation}\label{eq:RegAlg}
	A = \bigoplus\limits_{g \in G} g.
\end{equation}
Thus, $A \o g = g \o A = A$ in $\CC$ for $g \in G$. Hence, for any $ M \in \OO(\CC_{A}) $ and for any $ X \in \OO(\CC) $, we have
\begin{equation}\label{eq:Adj}
	[M:g \o X]_{\CC} = [A \o g\o  X: M]_{\CC_{A}}
	= [A \o  X: M]_{\CC_{A}}
	= [M: X]_{\CC}\,.
\end{equation}	

Now fix an $ M \in \OO(\CC_A) $ and let $ X \in \OO(\CC) $ be a simple subobject of $ M $ in $ \CC $. In other words, $ [M:X]_\CC \neq 0 $. Since $A \o X = \bigoplus_{g \in G} g \o X$ and $M$ is a direct summand of $A \o X$ in $\CC$, every simple subobject $Y$ of $M$ in $\CC$ is of the form $g \o X$ for some $g \in G$. In particular, for any $ Y \in \OO(\CC) $ such that $ [M:Y]_\CC \neq 0 $, we have $\dim_\CC(Y)=\dim_\CC(X)$ and $[M:Y]_\CC = [M:X]_\CC$ by \eqref{eq:Adj}. Hence, for any $n \in \BZ$, we have
\begin{equation}\label{eq:TraceThetaN}
	\begin{aligned}
	\Tr_{\CC}(\theta_M^n)
	&=
	\sum_{Y \in \OO(\CC)}
	\theta_Y^n[M:Y]_{\CC}\dim_{\CC}(Y)\\
	&=
	\sum_{
		\substack{
			Y \in \OO(\CC)\\
			[M:Y]_{\CC} \neq 0
		}
	}
	\theta_{Y}^{n} [M: Y]_{\CC} \dim_{\CC}(Y)\\
	&=
	[M:X]_{\CC}\dim_{\CC}(X)
	\sum_{
		\substack{
			Y \in \OO(\CC)\\
			[M:Y]_{\CC} \neq 0
		}
	}
	\theta_{Y}^{n}.
	\end{aligned}
\end{equation}

\begin{lemma}\label{lemma:GroupCase}
	Let $ \CC $ be a premodular category with ribbon isomorphism $\theta$, and $G$ a Tannakian subgroup of $\CC$. Then for any $M \in \OO(\CC_G)\setminus\OO(\CC_{G}^{0})$ and $ n\in \BZ $ relatively prime to $\ord(T_\CC)$,
	$$
 \Tr_{\CC}(\theta^n_M) = 0\,.
 $$
\end{lemma}

\begin{proof}
	By Lemma \ref{prop:TraceOfTheta} (a), for any $ M \in \OO(\CC_{G}) \setminus \OO(\CC_{G}^0) $, we have $ \Tr_{\CC}(\theta_M) = 0 $. Therefore, by Equation (\ref{eq:TraceThetaN}), we have
	$$
	\sum_{
	\substack{
		Y \in \OO(\CC)\\
		[M:Y]_{\CC} \neq 0
	}}
	\theta_{Y}
	=
	0,
	$$
	as the dimension of a simple object in a fusion category is not 0 and $ [M:Y]_{\CC} \neq 0$ by assumption.
	For any integer $n$ relatively prime to $\ord(T_\CC)$, there exists  $\sigma_n \in\Gal(\Qb/\Q)$ such that $\sigma_n(\theta_X) = \theta_{X}^n$ for all $ X \in \OO(\CC)$. We then have
	$$
	\sum_{
		\substack{
			X \in \OO(\CC)\\
			[M:X]_{\CC} \neq 0
	}}
	\theta_{X}^n
	=
	\sigma_n
	\left(
	\sum_{
	\substack{
		X \in \OO(\CC)\\
		[M:X]_{\CC} \neq 0
	}}
	\theta_{X}
	\right)
	=
	0,
	$$	
	which, together with Equation (\ref{eq:TraceThetaN}), proves the lemma.
\end{proof}

\begin{theorem}\label{th:deeq}
	Let $ \CC $ be a premodular category, and $G$ a Tannakian subgroup of $\CC$. Then for all $ n \in \BZ $ relatively prime to  $\ord(T_\CC)$,
	$$\tau_n(\mathcal{C})=|G|\tau_n(\CC_G^0).$$
	In addition, if $\tau_n(\mathcal{C})\neq0$, then
	$\xi_n(\CC_G^0)=\xi_n(\CC)$.
\end{theorem}

\begin{proof} Let $\theta$ be the ribbon isomorphism of $\CC$, and $A$ the regular algebra of the Tannakian subcategory $\AA$ equivalent to $\Rep(G)$. Lemmas \ref{lem:TotallyPositive} and \ref{lemma:GroupCase} imply
	\begin{equation*}
	\begin{aligned}
	\tau_n(\CC)
	&= \sum_{M \in \OO(\CC_{A})}
	\dim_{A}(M)\Tr_{\CC}(\theta_M^n)\\
	&= \sum_{M \in \OO(\CC_{A}^0)}
	\dim_{A}(M) \Tr_{\CC}(\theta^n_M)+
	\sum_{M \in \OO(\CC_{A})\setminus \OO(\CC_{A}^{0})}
	\dim_{A}(M)\Tr_{\CC}(\theta^n_M)\\
	&=
	\dim_\CC(A)\sum_{M \in \OO(\CC_{A}^0)}
	\theta_M^n \dim_{A}(M)^2+0\\
	&=\dim_\CC(A)\tau_n(\CC_A^0) = |G|\tau_n(\CC_G^0).
	\end{aligned}
	\end{equation*}
	The last assertion follows directly from the definition of higher central charges as $|G|$ is a positive integer.
\end{proof}

With the notations as above, recall that in this case, the de-equivariantization $(\AA')_G$ of the centralizer of $ \AA $ is the same as the category of local $A$-modules $ \CC_A^0 $.

\begin{example}
	Consider the category $\mathcal{C}:=\mathcal{C}(\mathfrak{so}_5,4)$ of rank 15 containing the Tannakian  subcategory $\mathcal{A}:=\mathrm{Rep}(\BZ_2)$ and has $\mathrm{ord}(T_\mathcal{C})=28$. The de-equivariantization $(\mathcal{A}')_{\BZ_2} = \CC^0_{\BZ_2}$ factors as $\mathcal{D}^{\boxtimes2}$, where $\mathcal{D}:=(\mathcal{C}(\mathfrak{sl}_2,5)'_\text{pt})^{\text{rev}}$ is a premodular category of rank 3. Let $d:=2\cos(\pi/7)$ and $\zeta=e^{\pi i/7}$. The dimensions of the nontrivial simple objects of $ \DD $ are $ d $ and $ d^2-1 $, and their twists are $ \zeta^{-2} $ and $ \zeta^4 $ respectively (cf. \cite{2018arXiv181009057S}).  There are 12 integers $ m $ such that $1\leq m<28$ and $\gcd(m,28)=1$. Therefore, $\tau_m(\mathcal{C})=2\tau^2_m(\DD)$ by Theorem \ref{th:deeq}, and we have
	\begin{equation}
	\xi_m(\mathcal{C})
	=
	\left\{
	\begin{array}{lcl}
	\zeta^4 & \mbox{if} & m=5,13,19,27; \\
	\zeta^6 & \mbox{if} & m=3,17; \\
	\zeta^8 & \mbox{if} & m=11,25; \\
	\zeta^{10} & \mbox{if} & m=1,9,15,23.
	\end{array}
	\right.
	\end{equation}
\end{example}

One would expect Theorem \ref{th:deeq} could be generalized to any ribbon algebra $A$ of a premodular category $ \CC $.  However, we are only able to do this for modular categories $\CC$ as the Galois group actions on their modular data can be applied.

\begin{theorem}\label{th:deek}
Let $\CC$ be a modular  category, $A\in\mathcal{C}$ a ribbon algebra, and $N=\ord(T_\CC)$.  Suppose $\1_\CC$ and $\1_A$ are respectively the monoidal units of $\CC$ and $\CC_A^0$. If $n\in\mathbb{Z}$ is relatively prime to $N$, then
\begin{equation}
	\theta_{\hat{\sigma}(\1_{\CC})}= \theta_{\hat{\sigma}(\1_A)}, \quad \xi_n(\CC_A^0)=\xi_n(\CC)\quad \text{and}\quad \tau_n(\CC_A^0)=\frac{\sigma(\dim_\CC(A))}{\dim_\CC(A)^2}\tau_n(\CC) \ne 0
\end{equation}
where $\sigma$ is any automorphism of $\Qb$ such that $\sigma(\zeta_N)=\zeta_N^{\tilde{n}}$ and $\tilde{n}$ is the multiplicative inverse of $n$ modulo $N$ as in Section \ref{sec:arith}.
\end{theorem}

\begin{proof}
By Theorem \ref{thm:Galois} and Lemma \ref{lem:known},

\begin{equation}\label{eq:deek}
\begin{aligned}
	\tau_n(\CC_A^0)
	&=
	\sigma(\tau_1(\CC_A^0))
	\frac{\dim(\CC_A^0)}{\sigma(\dim(\CC_A^0))}
	\theta^{n}_{\hat\sigma(\1_A)}\\
	&=
	\frac{\sigma(\tau_1(\CC))}{\sigma(\dim_\CC(A))}
	\frac{\dim(\CC_A^0)}{\sigma(\dim(\CC_A^0))}
	\theta^{n}_{\hat\sigma(\1_A)}\\
	&=
	\tau_n(\CC)
	\frac{\sigma(\dim(\CC))}{\dim(\CC)}
	\frac{\dim(\CC_A^0)}{\sigma(\dim_\CC(A)\dim(\CC_A^0))}
	\left(
		\frac
		{\theta_{\hat{\sigma}(\1_{A})}}
		{\theta_{\hat{\sigma}(\1_\CC)}}
	\right)^n\\
	&=
	\tau_n(\CC)
	\frac{\sigma(\dim(\CC))}{\dim(\CC)}
	\frac{\dim(\CC) \dim_\CC(A)^{-2}}{\sigma(\dim(\CC)\dim_\CC(A)^{-1})}
	\left(
	\frac
		{\theta_{\hat{\sigma}(\1_{A})}}
		{\theta_{\hat{\sigma}(\1_\CC)}}
	\right)^n\\
	&=
	\frac{\sigma(\dim_\CC(A))}{\dim_\CC(A)^2}
	\tau_n(\CC)
	\left(
	\frac
		{\theta_{\hat{\sigma}(\1_{A})}}
		{\theta_{\hat{\sigma}(\1_\CC)}}
	\right)^n.
\end{aligned}
\end{equation}

For any $\gamma \in \BC$ such that $\gamma^3=\xi_1(\CC)$, we also have $\gamma^3=\xi_1(\CC_A^0)$ by Corollary \ref{lem:known}.  Thus by Equation (\ref{eq:GammaAndTheta}),
\begin{equation*}
	\theta_{\hat{\sigma}(\1_{\CC})}
	=
	\frac{\gamma}{\sigma^2(\gamma)}
	=\theta_{\hat{\sigma}(\1_{A})},
\end{equation*}
hence $\theta_{\hat{\sigma}(\1_{\CC})}/\theta_{\hat{\sigma}(\1_A)}=1$ which can be substituted into (\ref{eq:deek}) to yield the third equality of the statement. Since $\CC$ is modular, $\t_n(\CC)\ne 0$ by Theorem \ref{thm:Galois} and hence $\tau_n(\CC_A^0)$. Since $\dim_\CC(A)$ is totally positive, we have $|\tau_n(\CC_A^0)|=\frac{\sigma(\dim_\CC(A))}{\dim_\CC(A)^2}
	|\tau_n(\CC)|$. Therefore, the equality $\xi_n(\CC_A^0)=\xi_n(\CC)$ follows directly from definition.
\end{proof}

Theorem \ref{th:deek} could be viewed as a generalization of the last statement of Theorem \ref{thm:GaussSumOfCenter}.  Indeed, if $\CC=\ZZ(\DD)$ for some spherical fusion category $\DD$, then  the \emph{Lagrangian algebra} $A=I(\1_\DD)$ is a ribbon algebra of $\CC$ where $I$ is the left adjoint of the forgetful functor from $\CC$ to $\DD$. Since $\dim_\CC(A)^2=\dim(\mathcal{C})$, $\dim(\CC_A^0) = 1$ and so $\CC_A^0 \simeq \Vec$ (cf. \cite[Proposition 5.8]{DMNO}). Hence Theorem \ref{th:deek} implies
\begin{equation*}
\tau_n(\mathcal{Z}(\mathcal{D}))=\dfrac{\dim_\CC(A)^2}{\sigma(\dim_\CC(A))}\tau_n(\mathrm{Vec})=\dfrac{\dim(\mathcal{D})^2}{\sigma(\dim(\mathcal{D}))}.
\end{equation*}

\end{section}


\begin{section}{Witt relations and central charges}\label{sec:witt}

Recall that each Witt equivalence class in $\mathcal{W}$ has a completely anisotropic (contains no nontrivial connected \'etale algebras) representative which is unique up to braided equivalence \cite[Theorem 5.13]{DMNO}.  For Witt equivalence classes in the unitary subgroup $\mathcal{W}_\text{un}\subset\mathcal{W}$, generated by equivalence classes of pseudounitary nondegenerate braided fusion categories, this completely anisotropic representative is unique up to ribbon equivalence, as there exists a \emph{unique} spherical structure with nonnegative dimensions for all objects (cf. \cite{ENO}).  By \cite[Lemma 5.27] {DMNO}, if two pseudounitary modular categories $\mathcal{C}$ and $\mathcal{D}$ are Witt equivalent, then $\xi_1(\mathcal{C})=\xi_1(\mathcal{D})$.  The goal of this section is to extend this result to higher central charges with the following theorem.

\begin{theorem}\label{thm:Witt}
Let $\mathcal{C}$ and $\mathcal{D}$ be pseudounitary modular tensor categories such that $[\mathcal{C}]=[\mathcal{D}]$.  If $n\in\mathbb{Z}$ is coprime to  $\mathrm{ord}(T_\mathcal{C})\cdot\mathrm{ord}(T_\mathcal{D})$, then $\xi_n(\mathcal{C})=\xi_n(\mathcal{D})$.
\end{theorem}

\begin{proof}
Recall \cite[Corollary 5.9]{DMNO} that $\mathcal{C}$ is Witt equivalent to $\mathcal{D}$ if and only if there exists a fusion category $\mathcal{A}$ such that $\mathcal{Z}(\mathcal{A})\simeq\mathcal{C}\boxtimes\mathcal{D}^\mathrm{rev}$ is a braided equivalence. Since $\mathcal{Z}(\mathcal{A})$ is the Deligne product of pseudounitary categories, it is also pseudounitary. Moreover, since
$$
\dim(\AA)^2=\dim(\mathcal{Z}(\mathcal{A})) = \FPdim(\ZZ(\AA))=\FPdim(\AA)^2,
$$
$\mathcal{A}$ is also pseudounitary. The assumption $\gcd(n,\mathrm{ord}(T_\mathcal{C})\mathrm{ord}(T_\mathcal{D}))=1$ implies $\xi_n(\mathcal{C})$ and $\xi_n(\mathcal{D})$ are well-defined by Theorem \ref{thm:Galois}. By Theorem \ref{thm:GaussSumOfCenter},  $\xi_n(\mathcal{Z}(\mathcal{A}))=1$ as $\ord(T_{\ZZ(\AA)})=\lcm(\ord(T_\CC),\ord(T_{\DD}))$. Thus, by Lemmas \ref{lem1} and \ref{lem:mult}, we find
$$
\xi_n(\CC) \ol{\xi_{n}(\DD)}=1\,.
$$
By Corollary \ref{cor:RootOfUnity}, $\xi_{n}(\CC)$ and $\xi_{n}(\DD)$ are root of unity, and the result follows.
\end{proof}

The structure theorem for the classical Witt group of quadratic form $\WQ$ \cite{Schar} coincides with that of the pointed Witt group $\mathcal{W}_\mathrm{pt}$ \cite[Proposition 5.17]{DMNO}. In particular, \begin{equation}\label{eq:WittDecomp}
    \mathcal{W}_\text{pt}\cong\bigoplus_{\text{primes }p}\mathcal{W}_\text{pt}(p)
\end{equation}
where $\mathcal{W}_\text{pt}(p)$ consists of the equivalence classes of all pointed modular categories with the fusion rules of abelian $p$-groups. It is well-known that $ \mathcal{W}_{\mathrm{pt}}(2) \cong \mathbb{Z}_8 \oplus \mathbb{Z}_2$, $ \mathcal{W}_{\mathrm{pt}}(p) \cong  \mathbb{Z}_4$ for $ p \equiv 3\ (\mathrm{mod}\ 4) $, and $ \mathcal{W}_{\mathrm{pt}}(p) \cong \mathbb{Z}_2 \oplus \mathbb{Z}_2$ for $ p \equiv 1\  (\mathrm{mod}\ 4) $.

In Example \ref{wrong2}, we have demonstrated the application of the first central charge to distinguish the generator of $\WWp(p)$ when $p$ is an odd prime. In fact, the higher central charges can distinguish all element of $\WWp(p)$ for any prime $p$. We demonstrate this application in the following example.
\begin{example}
The group $\WWp(2)$ is generated by $\CC=\CC(\BZ_4, q)$ and $\DD=\CC(\BZ_4, q)\boxtimes\CC(\BZ_2, q')$, where $q(1)=\zeta_8$ and $q'(1)=\zeta_4\inv$. One can compute directly that
$$
\xi_1(\CC)=\zeta_8,\quad \xi_3(\CC)=\zeta_8^3, \quad \xi_1(\DD)= 1, \quad \xi_3(\DD)=-1.
$$
Denote by $\mathcal{E}_{a,b} = \CC^{\boxtimes a} \boxtimes \DD^{\boxtimes b}$ for any non-negative integers $a, b$, and $\WW'$ the subgroup of $\WWp(2)$ generated by the Witt equivalence classes of $\CC$ and $\DD$. We  find
$$
\xi([\mathcal{E}_{a,b}]) := (\xi_1(\mathcal{E}_{a,b}), \xi_3(\mathcal{E}_{a,b}))=( \zeta_8^a, (-1)^{b}\zeta_8^{3a})\,.
$$
In particular, $\xi: \WW' \to \langle \zeta_8 \rangle \times \langle \zeta_8\rangle$ defines a group homomorphism with its image isomorphic to $\BZ_8 \oplus \BZ_2$. Since $\WWp(2)$ is of order 16,  $[\CC]$ and $[\DD]$ are generators of $\WWp(2)$ by Theorem \ref{thm:Witt}.
\end{example}

Outside of $\mathcal{W}_\mathrm{pt}$ one can use higher central charges to differentiate various Witt equivalence classes.

\begin{example}
In this example, we use the formulas for the modular data of $\mathcal{C}(\mathfrak{g}_2,k)$ as in \cite[Sections 2.3.4]{schopieray2}.  Set $\mathcal{C}:=\mathcal{C}(\mathfrak{g}_2,8)^{\boxtimes5}$ and $\mathcal{D}:=\mathcal{C}(\mathfrak{g}_2,11)^{\boxtimes10}$.  One computes
\begin{equation}
\xi_1(\mathcal{C})=\exp(-\pi i/3)=\xi_1(\mathcal{D}),
\end{equation}
but $\mathcal{C}$ and $\mathcal{D}$ are not Witt equivalent.  To see this, note that $\mathrm{ord}(T_\mathcal{C}) = 36$, and $\mathrm{ord}(T_\mathcal{D}) = 45$ \cite{schopieray2}. By direct computation, we have
$\xi_{13}(\mathcal{C})=1$ and $\xi_{13}(\mathcal{D})=-1$. Hence $[\mathcal{C}]\neq[\mathcal{D}]$ by Theorem \ref{thm:Witt}.  All Witt group relations amongst classes $[\mathcal{C}(\mathfrak{g}_2,k)]$ were classified in \cite{2018arXiv181009057S} corroborating this result.	
\end{example}

The following proposition implies the converse of Theorem \ref{thm:Witt} does not hold.

\begin{proposition} Let $ \mathcal{C} $ be a modular category and $ N:= \ord(T_{\mathcal{C}})$. For any integer $ k $ relatively prime to $ N $,
$$
\xi_k(\mathcal{C}^{\boxtimes 4N})
=
1.
$$
\end{proposition}
\begin{proof}
By Theorem \ref{thm:Galois}, $ \alpha_k(\mathcal{C}) $ and $ \xi_k(\mathcal{C}) $ are well-defined. By Lemma \ref{lem:mult}, Corollary \ref{cor:RootOfUnity}, and Lemma \ref{lem1}, we have
e have
$$
\xi_k(\mathcal{C}^{\boxtimes 4N})
=
\xi^2_k(\mathcal{C}^{\boxtimes 2N})
=
\alpha_k(\mathcal{C}^{\boxtimes 2N})
=
\alpha_k^{2N}(\mathcal{C})
=
\sigma(\alpha^{2N}_1(\mathcal{C}))\theta_{\hat\sigma(\1)}^{2N}
=
1. \qedhere
$$
\end{proof}

\begin{example}
Let $ \CC $ be a modular category such that its Witt equivalence class $ [\CC]\in\WW $ is of infinite order. The existence of such categories is discussed in \cite[Example 6.4]{DMNO}. By the above proposition, for any $ k $ coprime to $ N $, $ \xi_k(\mathcal{C}^{\boxtimes 4N}) = 1 = \xi_k(\Vec) $, but $ [\mathcal{C}^{\boxtimes 4N}] \neq [\Vec] $ in $ \mathcal{W} $ by our assumption on $ \mathcal{C} $. Hence, the above proposition provides counter-examples to the converse of Theorem \ref{thm:Witt}.
\end{example}

\end{section}
\begin{section}{Questions}

\par The higher central charges of pointed modular categories can be computed by using the structure of $\WWp$  (Section \ref{sec:witt}). There is an explicit formula for $\xi_1(\mathcal{C}(\mathfrak{g},k))$ based only on $\dim(\mathfrak{g})$, $k$, and the dual Coxeter number of $\mathfrak{g}$ \cite[Equation 7.4.5]{BakalovKirillov}.  As $\xi_n$ is often undefined when $n\neq\pm1,0$ it is unclear whether a similar general formula (i.e. \!without reference to Galois automorphisms in Theorem \ref{thm:Galois}) for higher central charges exists.

\begin{question}
Does an explicit formula exist for $\xi_n(\mathcal{C}(\mathfrak{g},k))$ for $n\neq\pm1,0$ using only the input data of $n$, $\mathfrak{g}$, and $k$?
\end{question}

\par By Example \ref{ex:zero_gauss_sum}, there exist modular categories $\DD$ with some of its higher Gauss sums being 0. Thus, if $\mathcal{C}$ is a modular category containing a modular subcategory $\DD$ such that $\t_a(\DD)=0$ for some integer $a >1$, then $\t_{a}(\CC)=0$ as $\CC\simeq\DD\boxtimes\DD'$.
\begin{question}
Are there necessary and sufficient conditions for higher Gauss sums of a premodular (or modular) category to vanish (hence higher central charges are undefined)?
\end{question}

\end{section}

\begin{acknowledgement}
The second author would like to thank Victor Ostrik for his suggestion to explore the notion of higher Gauss sums. Both the second and third authors would like to thank MSRI (Summer Graduate School 791) for providing the opportunity to initiate this collaboration. The third author would like to thank Thomas Kerler and James Cogdell, and the first author would like to thank Ling Long  for fruitful discussions.
\end{acknowledgement}
\bibliographystyle{abbrv}

\end{document}